\documentclass[manuscript,screen]{acmart}
\usepackage{subfig}
\AtBeginDocument{%
  \providecommand\BibTeX{{%
    \normalfont B\kern-0.5em{\scshape i\kern-0.25em b}\kern-0.8em\TeX}}}

\setcopyright{acmcopyright}
\copyrightyear{2018}
\acmYear{2018}
\acmDOI{XXXXXXX.XXXXXXX}

\usepackage{dependencies}
\begin{document}

\title{Improvements to Quantum Interior Point Method for Linear Optimization}

\author{Mohammadhossein Mohammadisiahroudi}
\email{mom219@lehigh.edu}
\affiliation{%
  \institution{Lehigh University}
  \state{PA}
  \country{USA}
}
\email{mom219@lehigh.edu}
\author{Zeguan Wu}
\affiliation{%
  \institution{Lehigh University}
  \state{PA}
  \country{USA}}

\author{Brandon Augustino}
\affiliation{%
  \institution{Massachusetts Institute of Technology}
  \state{MA}
  \country{USA}
}

\author{Arielle Carr}
\affiliation{%
 \institution{Lehigh University}
 \state{PA}
 \country{USA}}

\author{Tamás Terlaky}
\affiliation{%
  \institution{Lehigh University}
 \state{PA}
 \country{USA}}

\renewcommand{\shortauthors}{Mohammadisiahroudi, et al.}

\begin{abstract}
  Quantum linear system algorithms (QLSA) have the potential to speed up Interior Point Methods (IPM).
  However, a major challenge is that QLSAs are inexact and sensitive to the condition number of the coefficient matrices of linear systems. This sensitivity is exacerbated when the Newton systems arising in IPMs converge to a singular matrix.
  Recently, an Inexact Feasible Quantum IPM (IF-QIPM) has been developed that addresses the inexactness of QLSAs and, in part, the influence of the condition number using iterative refinement.
  However, this method requires a large number of gates and qubits to be implemented. 
  Here, we propose a new IF-QIPM using the normal equation system, which is more adaptable to near-term quantum devices.
  To mitigate the sensitivity to the condition number, we use preconditioning coupled with iterative refinement to obtain better gate complexity.
  Finally, we demonstrate the effectiveness of our approach on IBM Qiskit simulators. 
\end{abstract}



\received{...}
\received[revised]{...}
\received[accepted]{..}

\maketitle

\section{Introduction}
Mathematical optimization problems arise in many fields and their solution yields significant computational challenges. 
Researchers have attempted to develop quantum optimization algorithms, such as the Quantum Approximation Optimization Algorithm (QAOA) for unconstrained quadratic binary optimization problems
\cite{farhi2014quantum}, and a quantum subroutine for simplex algorithm \cite{nannicini2022fast}.
Another class of quantum algorithms are Quantum Interior Point Methods (QIPMs) \cite{Augustino2021_quantum, kerenidisParkas2020_quantum, mohammadisiahroudi2022efficient}, which are hybrid-classical IPMs that use QLSAs to solve the Newton system at each IPM iteration.
Before reviewing prior work on QIPMs for Linear Optimization problems (LOP), we provide the necessary definitions, fundamental results, and properties.
\begin{definition}[LOP: Standard Formulation]\label{def: LO-standard}
For $b\in \mathbb{R}^m$, $c\in \mathbb{R}^n$, and matrix $A\in \mathbb{R}^{m\times n}$ with $\rank(A)=m \leq n$, the LOP is defined as
\begin{equation*} 
    (\text{P})\quad
    \begin{aligned}
    \min\  c^T&x, \\
    {\rm s.t. }\;\;
    Ax &= b, \\
    x &\geq 0,
    \end{aligned}
    \qquad \qquad (\text{D}) \quad
    \begin{aligned}
    \max \  b^Ty,\ \ & \\
    {\rm s.t. }\;\;
    A^Ty +&s = c,\\
    &s \geq 0,
    \end{aligned}
\end{equation*}
where $x\in \Rmbb^n$ is the vector of primal variables, and $y\in\Rmbb^m$, $s\in\Rmbb^n$ are vectors of the dual variables.
Problem (P) is called primal problem and problem (D) is called dual problem. 
\end{definition}
The set of feasible primal-dual solutions is defined as
\begin{equation*}
\mathcal{PD}=\left\{(x,y,s)\in\mathbb{R}^{n} \times \mathbb{R}^m\times\mathbb{R}^n |\ Ax=b,\ A^Ty+s=c, \ (x,s)\geq0\right\}.
\end{equation*}
Then, the set of all feasible interior solutions is  defined as 
\begin{equation*}
\mathcal{PD}^0=\left\{(x,y,s)\in\mathcal{PD}\ |\ (x,s)>0\right\}.
\end{equation*}
In this work, we assume $\mathcal{PD}^0$ is not empty.
By the Strong Duality theorem, optimal solutions exist and belong to the
set $\mathcal{PD}^*$ defined as
\begin{equation*}
\mathcal{PD}^*=\left\{(x,y,s)\in\mathcal{PD} \ | \ x^Ts=0\right\}.
\end{equation*}
Let $\zeta \geq 0$, then the set of $\zeta$-optimal solutions is defined as
\begin{equation*}
\mathcal{PD}_{\zeta} = \left\{(x,y,s)\in\mathcal{PD} \ \Big|\ \frac{x^Ts}{n}\le \zeta \right\}.
\end{equation*}
In each step of IPMs, a Newton system is solved to determine the Newton step. There are four approaches:
\begin{enumerate}
    \item Full Newton System
\begin{equation}\label{eq: FNS}\tag{FNS}
    \begin{bmatrix}
0&A&0\\
A^T&0&I\\
0&S&X
\end{bmatrix}\begin{bmatrix}
\Delta y\\
\Delta x\\
\Delta s
\end{bmatrix}=\begin{bmatrix}
0\\0\\
\beta \mu e- Xs
\end{bmatrix};
\end{equation}
\item Augmented System
\begin{equation}\label{eq: AS}\tag{AS}
    \begin{bmatrix}
0&A\\
A^T&-D^{-2}
\end{bmatrix}\begin{bmatrix}
\Delta y\\
\Delta x
\end{bmatrix}=\begin{bmatrix}
0\\
s- \beta \mu X^{-1}e
\end{bmatrix};
\end{equation}
\item Normal Equation System
\begin{equation}\label{eq: NES}\tag{NES}
AD^2A^T \Delta y= Ax- \beta \mu A S^{-1}e;
\end{equation}
\item Orthogonal Subspaces System
\begin{equation}\label{eq: OSS}\tag{OSS}
  \begin{bmatrix}
-XA^T&SV
\end{bmatrix}\begin{bmatrix}
\Delta y\\
\lambda
\end{bmatrix}=\beta \mu e- Xs,
\end{equation}
\end{enumerate}
where $X=\text{diag}(x)$, $S=\text{diag}(s)$, $D=S^{-1/2}X^{1/2}$, $\mu=\frac{x^Ts}{n}, \beta$, $0<\beta<1$, and $e$ is an all-one vector. Further, the columns of $V$ form a basis for the null space of $A$. 
\begin{table}[h]
    \centering
\begin{tabular}{ |c|c|c|c|c| } 
 \hline
 System & Size of system & Symmetric & Positive Definite&  Rate of Condition Number\\ 
 \hline
\ref{eq: FNS} & $2n+m$& \xmark& \xmark&  $\Ocal\big(\frac{1}{\mu^2}\big)$\\
\ref{eq: AS}& $n+m$& \cmark& \xmark&  $\Ocal\big(\frac{1}{\mu^2}\big)$\\
\ref{eq: NES} & $m$ & \cmark& \cmark& $\Ocal\big(\frac{1}{\mu^2}\big)$\\
\ref{eq: OSS}  & $n$ & \xmark& \xmark&  $\Ocal\big(\frac{1}{\mu}\big)$\\
\hline
\end{tabular} 
\caption{Characteristics of Coefficient Matrix of Different Newton Systems}
\label{tab: systems}
\end{table}

Table \ref{tab: systems} shows that NES has a smaller size since in most practical LO problems $m<<n$. 
In addition, its symmetric positive definite coefficient matrix is favorable since classically it can be solved faster with Cholesky factorization or conjugate gradient. 
It is also more adaptable to QLSAs since QLSAs are able to solve linear systems with a Hermitian matrix. 
To solve linear systems whose matrix is not Hermitian, like OSS, they must be embedded in a bigger system with a Hermitian matrix. 
Thus, NES has a better structure compared to others, however, its condition number grows at a faster rate than the one of the OSS.
The key takeaway is that inexact solutions of NES, FNS, and AS may lead to infeasibility, whereas inexact solutions to OSS remain in the feasible region \cite{IF-QIPMforLO}. 

The prevailing approach to solve LOPs is IPMs, in which the Newton direction is calculated by solving NES using Cholesky factorization \cite{Roos2005_Interior}. 
Thus, although IPMs enjoy a fast convergence rate, the cost per iteration of IPMs is considerably high when applied to large-scale LOPs. 
In an effort to reduce the per-iteration cost of IPMs, inexact infeasible IPMs (II-IPMs) were proposed, in which the Newton system is solved with an iterative method, e.g., using conjugate gradient (CG) methods \cite{Monteiro2003_Convergence, Gondzio2009_Convergence}. 
Inexact linear systems algorithms like CG exhibit favorable dependence on dimension compared to factorization methods and are able to exploit sparsity patterns present in the Newton system. 
The rub is that these inexact approaches depend on a condition number bound, which could pose a challenge.
To tackle the ill-conditioned Newton system, they used the so-called maximum weight basis (MWB) preconditioner \cite{monteiro2004uniform}. 

QIPMs were first proposed by Kerenidis and Prakash \cite{kerenidisParkas2020_quantum}, who sought to decrease the cost per iteration by classically estimating the Newton step through the use of a QLSA and quantum state tomography. 
Adopting this approach, Casares and Martin-Delgado \cite{Casares2020_quantum} developed a predictor-correcter QIPM for LO.
However, these algorithms were proposed and analyzed using an \textit{exact} IPM framework, which is invalid because the use of quantum subroutines naturally introduces noise into the solution and leads to \textit{inexactness} in the Newton step. Specifically, without further safeguards this inexactness means that the sequence of iterates generated the algorithms in \cite{Casares2020_quantum, kerenidisParkas2020_quantum} may leave the feasible set, and so convergence cannot be guaranteed. 

To address these issues, Augustino et al.~\cite{Augustino2021_quantum} proposed an inexact-infeasible QIPM (which closely quantized the II-IPM of \cite{Zhou2004_Polynomiality}) and a novel \textit{inexact-feasible} QIPM using OSS. The latter algorithm was shown to solve LOPs to precision $\epsilon \in (0,1)$ using at most 
$$\widetilde{\Ocal}_{n, \kappa, \frac{1}{\epsilon}} \left( n^2 \frac{\kappa^2}{\epsilon} \right)$$
QRAM queries and $\widetilde{\Ocal}_{n, \kappa, \frac{1}{\epsilon}} \left( n^{2.5} \right)$ arithmetic operations, where $\kappa$ is an upper bound on the Newton system coefficient matrices that arise over the run of the algorithm. 

Mohammadisiahroudi et al.~\cite{mohammadisiahroudi2022efficient, IF-QIPMforLO} specialized the algorithms in \cite{Augustino2021_quantum} to LO and used iterative refinement techniques to exponentially improve the dependence of the algorithms in \cite{Augustino2021_quantum} on precision and the condition number bound. In particular, \cite{mohammadisiahroudi2022efficient} developed an inexact-infeasible QIPM (II-QIPM), which addresses the inexactness of QLSA, with $$\widetilde{\Ocal}_{n, \kappa_A, \omega}(n^4L\|A\|^4\omega^{4}\kappa_A^2)$$ complexity, where $\omega$ is an upper bound for norm of optimal solution. \cite{IF-QIPMforLO} improved this complexity by developing a short-step IF-QIPM for LOPs with complexity  
$$\widetilde{\Ocal}_{n, \kappa_A, \omega}(n^{2.5}L\|A\|^2\omega^2\kappa_A).$$

Note that the use of iterative refinement techniques indirectly led to another improvement in the complexity, reducing the dependence on a condition number bound $\kappa$ for the intermediate Newton systems with the condition number $\kappa_A$ of the input matrix $A$. 
IF-QIPMs built on similar techniques have also been developed for linearly constrained quadratic optimization problems in \cite{wu2023inexact} and second-order cone optimization problems in \cite{augustino2021inexact}. 

In this paper, we propose an IF-QIPM using a modified normal equation system that has a better structure, with a smaller symmetric positive definite coefficient matrix, and enables us to use classical preconditioning techniques to mitigate the effect of the condition number. 
In addition, we explore how the preconditioning technique augmented with iterative refinement can help with condition number issues. 

The rest of this paper is structured as follows. 
In Section~\ref{sec: MNES},  a modified NES is utilized to produce an inexact but feasible Newton step, and a short-step Inexact Feasible IPM is developed. 
Section~\ref{sec: IF-QIPM} explores how we use QLSA to solve the modified NES system in order to develop an IF-QIPM.  
In Section~\ref{sec: condition}, an iterative refinement method coupled with preconditioning is developed to address the impacts of the condition number on the complexity.
Finally, numerical experiments using the IBM Qiskit simulator are carried out in Section~\ref{sec: numerical}, and Section~\ref{sec: conclusion} concludes the paper.

\section{Inexact-Feasible Newton Step using NES}\label{sec: MNES}
To compute the Newton step, we need to determine $(\Delta x, \Delta y, \Delta s)$ such that
\begin{equation}\label{eq:newton system}
\begin{aligned}
    A\Delta x&=0,\\
    A^T \Delta y +\Delta s&=0,\\
    X\Delta s+S\Delta x&=\beta \mu e-Xs.
\end{aligned} 
\end{equation}
As discussed in the introduction, we are interested in using NES as
it has a smaller symmetric positive definite matrix, favorable for both quantum and classical linear system solvers. First note that an exact solution $\Delta y$ to NES satisfies
\begin{equation}\label{e:NES}
AD^2A^T \Delta y= Ax- \beta \mu A S^{-1}e. 
\end{equation}
Having obtained $\Delta y$, we then compute $\Delta s$ and $\Delta x$ using as follows:
\begin{subequations}
    \begin{align}
    \Delta s&=-A^T \Delta y ,\label{eq:deltaS}\\
    \Delta x&=\beta \mu S^{-1}e-x-D^2\Delta s.\label{eq:deltaX}
\end{align}
\end{subequations}
Now, when system \eqref{e:NES} is solved inexactly, the resulting solution $\overline{\Delta y}$ satisfies
\begin{equation*}
AD^2A^T \overline{\Delta y}= Ax- \beta \mu A S^{-1}e+r, \end{equation*}
where $r$ is the residual as $\|Ax- \beta \mu A S^{-1}e- AD^2A^T \overline{\Delta y}\|$. In place of (\ref{eq:deltaS}) and (\ref{eq:deltaX}), we now have
\begin{equation*}
\begin{aligned}
    A\Delta x&=r,\\
    A^T \Delta y +\Delta s&=0,\\
    X\Delta s+S\Delta x&=\beta \mu e-Xs.
\end{aligned} 
\end{equation*}
While dual feasibility is preserved, the same cannot be said for the primal. In order to preserve primal feasibility using inexact solutions to \eqref{e:NES}, one can alternatively solve
\begin{align*}
    \Delta s&=-A^T \Delta y ,\\
    \Delta x&=\beta \mu S^{-1}e-x-D^2\Delta s-v,
\end{align*}
where $Av=r$. Updating $\Delta x$ in this way, we \textit{correct} primal infeasibility, and one can verify that 
\begin{equation*}
\begin{aligned}
    A\Delta x&=0,\\
    A^T \Delta y +\Delta s&=0,\\
    X\Delta s+S\Delta x&=\beta \mu e-Xs+r',
\end{aligned} 
\end{equation*}
where $r'=-Sv$. Next, we describe two procedures to calculate $v$ efficiently.

\textbf{Procedure A.}
Since $A$ has full row rank, we can calculate $$\hat{A}=A^T(AA^T)^{-1},$$ 
as a pre-processing step before the IPM starts. Then, in each iteration, we calculate $v=\hat{A}r$ using classical matrix-vector products. To recover the convergence analysis of \citep{IF-QIPMforLO}, the residual must satisfy $\|r'\|\leq\eta \mu$ for $\eta \in [0,1)$. One can show that this requirement amounts to
$$\|r\|\leq \eta \frac{\mu}{\|s\|_{\infty}\sigma_{\max}},$$
where $\sigma_{max}$ is the maximum singular value of $A$. Since $\|s\|_{\infty}$ and $\sigma_{\max}$ can be exponentially large, this residual bound can be unacceptably small.

\textbf{Procedure B.}
Letting $\Bhat$ be an arbitrary basis for matrix $A$, we can calculate
$$v=(v_{\Bhat},v_N)=(A_{\Bhat}^{-1}r,0).$$
It is straightforward to verify that $Av=r$. Now, we show that this procedure coupled with an appropriate modification of the NES leads to a favorable residual bound. 

Since $A$ has full row rank, one can choose an arbitrary basis $\Bhat$, and subsequently calculate $A_{\Bhat}^{-1}$, $\Ahat=A_{\Bhat}^{-1}A$, and $\bhat=A_{\Bhat}^{-1}b$. These steps require $\Ocal(m^2n)$ arithmetic operations and take place only once prior to the first iteration of IPM. The cost of this pre-processing can be reduced by leveraging the structure of $A$. For example, if the problem is in the canonical form, there is no need for this pre-processing.
In this paper, we neglect the pre-processing cost, since it can be avoided by using the following reformulation.
\begin{equation*} 
    \begin{aligned}
    \min\  c^T&x, \\
    {\rm s.t. }\;\;
    Ax +u&= b, \\
    -Ax +u'&= -b, \\
    x,u,u' &\geq 0.
    \end{aligned}
\end{equation*}
This is a standard LOP, but its interior is empty. This issue is remedied upon using the self-dual embedding model \cite{Ye1994_iteration} and we refer the readers to \cite{IF-QIPMforLO} for details.
While this formulation does not require calculation, the price one pays for this case is using a larger system. In the rest of this paper, we assume that we are working with the preprocessed problem with input data $(\hat{A},\hat{b},c)$.

Now, we can modify system \eqref{eq: NES} with coefficient matrix $M^k=A(D^k)^2A^T$ and right-hand side vector $\sigma^k =b-\beta\mu^kA(S^k)^{-1}e$ to 
\begin{equation*}\tag{MNES}\label{eq: modified normal equation}
    \Mhat^k z^k=\hat{\sigma}^k
\end{equation*}
where
\begin{align*}
    \Mhat^k&=(D_{\Bhat}^{k})^{-1}A_{\Bhat}^{-1}M^k((D_{\Bhat}^{k})^{-1}A_{\Bhat}^{-1})^T=(D_{\Bhat}^{k})^{-1}\Ahat (D^k)^2\Ahat^T(D_{\Bhat}^{k})^{-1},\\
\hat{\sigma}^k&=(D_{\Bhat}^{k})^{-1}A_{\Bhat}^{-1}\sigma^k=(D_{\Bhat}^{k})^{-1}\bhat-\beta_1 \mu^k (D_{\Bhat}^{k})^{-1}\Ahat(S^k)^{-1}e.
\end{align*}
We use the following procedure to find the Newton direction by solving \eqref{eq: modified normal equation} inexactly with QLSA+QTA.
\begin{enumerate}
    \item []\textbf{Step 1.} Find $\tilde{z}^k$ such that $\Mhat^k\tilde{z}^k=\hat{\sigma}^k+\rhat^k$ and $\|\rhat^k\|\leq\frac{\eta}{\sqrt{1+\theta}}\sqrt{\mu^k}$.
    \item []\textbf{Step 2.} Calculate $\widetilde{\Delta y}{}^k=((D_{\Bhat}^{k})^{-1}A_{\Bhat}^{-1})^T\tilde{z}^k$.
    \item []\textbf{Step 3.} Calculate $v^k=(v^k_{\Bhat},v^k_{\Nhat})=(D_{\Bhat}^{k}\rhat^k,0)$.
    \item []\textbf{Step 4.} Calculate $\widetilde{\Delta s}{}^k=c-A^Ty^k-s^k-A^T\widetilde{\Delta y}{}^k$.
    \item []\textbf{Step 5.} Calculate $\widetilde{\Delta x}{}^k=\beta_1\mu^k (S^{k})^{-1}e-x^k-(D^k)^2\widetilde{\Delta s}{}^k-v^k$.
\end{enumerate}
It is noteworthy that, this modification technique is similar to MWB preconditioning techniques of \cite{Monteiro2003_Convergence, Gondzio2009_Convergence}. One major difference is that we modify the NES for a feasible IPM setting, although others apply it for infeasible IPMs. In addition, we preprocess the data initially, before starting IPM although in preconditioning, one needs to do the modification, calculating the precondition with $\Ocal(n^3) cost$, in each iteration. Thus, we show that the complexity of our approach has better dimension dependence $\Ocal(n^{2.5})$ although the complexity of other infeasible approaches with preconditioning has $\Ocal(n^{5})$ dimension dependence. In Section~\ref{sec: condition}, we explore how quantum computing can speed up the preconditioning part.

\begin{lemma}\label{lem: newton form}
For the Newton direction $(\widetilde{\Delta x}{}^k,\widetilde{\Delta y}{}^k,\widetilde{\Delta s}{}^k)$, we have
\begin{equation} \label{eq:approximate modified newton system}
    \begin{aligned}
    A\widetilde{\Delta x}{}^k&=0,\\
    A^T \widetilde{\Delta y}{}^k +\widetilde{\Delta s}{}^k&=0,\\
    X^k\widetilde{\Delta s}{}^k+S^k\widetilde{\Delta x}{}^k&=\beta_1 \mu^k e-X^ks^k+{r'}^k.
\end{aligned}
\end{equation}
where ${r'}^k=-S^kv^k$.
\end{lemma}
\begin{proof}
For the Newton direction $(\widetilde{\Delta x}{}^k,\widetilde{\Delta y}{}^k,\widetilde{\Delta s}{}^k)$, one can verify that $\Mhat^k\tilde{z}^k=\hat{\sigma}^k+\rhat^k$ implies
\begin{align*}
    M^k\widetilde{\Delta y}{}^k&=\sigma^k+A_{\Bhat}D_{\Bhat}^{k}\rhat^k.
\end{align*}
For the first equation of \eqref{eq:approximate modified newton system}, we can write
\begin{align*}
    A\widetilde{\Delta x}{}^k=&A(\beta_1\mu^k (S^{k})^{-1}e-x^k-(S^{k})^{-1}X^k\widetilde{\Delta s}{}^k-v^k)\\
    =&A(\beta_1\mu^k (S^{k})^{-1}e-x^k-(S^{k})^{-1}X^k(-A^T\widetilde{\Delta y}{}^k)-v^k)\\
    =&\beta_1\mu^k A(S^{k})^{-1}e-Ax^k+A(S^{k})^{-1}X^k A^T\widetilde{\Delta y}{}^k-Av^k\\
    =&\beta_1\mu^k A(S^{k})^{-1}e-Ax^k+\sigma^k+A_{\Bhat}D_{\Bhat}^{k}\rhat^k-Av^k\\
    =&0 .
\end{align*}
The second and third equations of \eqref{eq:approximate modified newton system} are obtained from Steps 4 and 5.
\end{proof}
To have a convergent IPM, we need $\|{r'}^k\|_{\infty}\leq \eta \mu^k$, where $0\leq\eta< 1$ is an enforcing parameter. The next lemma gives an analogous residual bound for the modified NES.
\begin{lemma}\label{lem: residual bound}
For the Newton direction $(\widetilde{\Delta x}{}^k,\widetilde{\Delta y}{}^k,\widetilde{\Delta s}{}^k)$, if the residual $\|\rhat^k\|_{\infty}\leq\frac{\eta}{\sqrt{1+\theta}} \sqrt{\mu^k}$, then $\|{r'}^k\|_{\infty}\leq \eta \mu^k$.
\end{lemma}
\begin{proof}
As $(x^k,y^k,s^k)\in \Ncal(\theta)$, we have
$$(1-\theta)\mu^k\leq x_is_i\leq(1+\theta)\mu^k.$$
Thus, $$\|(S^k_{\Bhat})^{1/2}(X_{\Bhat}^{k})^{1/2}\|_{\infty}=\max_{i\in\Bhat} \sqrt{x_is_i}\leq \max_{i=1}^n \sqrt{x_is_i}\leq \sqrt{(1+\theta)\mu^k}.$$
Now we can conclude that
\begin{align*}
    \|{r'}^k\|_{\infty}=\|S^kv^k\|_{\infty}=\|S^k_{\Bhat}v^k_{\Bhat}\|_{\infty}=\|S^k_{\Bhat}D_{\Bhat}^{k}\rhat^k\|_{\infty}\leq \|(S^k_{\Bhat})^{1/2}(X_{\Bhat}^{k})^{1/2}\|_{\infty}\|\rhat^k\|_{\infty}\leq\sqrt{(1+\theta)\mu^k}\|\rhat^k\|_{\infty}\leq\eta\mu^k.
\end{align*}
\end{proof}
In the next subsection, we analyze the complexity of QLSA+QTA to solve the \eqref{eq: modified normal equation}.
\subsection{Solving MNES by QLSA+QTA}
The first QSLA was proposed by Harrow, Hassidim and Loyd~\cite{Harrow2009_quantum} and known as HHL algorithm. It takes as input a sparse, Hermitian matrix $M$, and prepares a state $\ket{z} = \ket{M^{-1} \sigma}$ that is proportional to the solution of the linear system $ Mz = \sigma$.
Let $\kappa_M$ denote the condition number of $M$. 
The complexity of the HHL algorithm is $\Tilde{\Ocal}_{d} ({\tau^2\kappa_M^2}/{\epsilon})$, where $d$ is the dimension of the problem, $\tau$ is the maximum number of non-zeros found in any row of $M$, $\epsilon$ is the target bound on the error, and the $\tilde{\Ocal}$ notation suppresses the polylogarithmic factors in the "Big-O" notation in terms of the subscripts. 
This complexity bound shows a speed-up w.r.t. dimension, although it depends on an upper bound for the condition number $\kappa_M$ of the coefficient matrix.
Following a number of improvements to HHL algorithm~\cite{ambainis2012variable, wossnig2018quantum, vazquez2022enhancing, childs2017quantum}, the current state-of-the-art QLSA is attributed to Charkraborty et al.~\cite{Blockencoding}, who use variable-time \textit{amplitude estimation} and so-called \textit{block-encoded} matrices, while HHL algorithm uses the \textit{sparse-encoding} model \cite{Harrow2009_quantum}. 
The block-encoding model was formalized in \cite{low2019hamiltonian}, and it assumes that one has access to unitaries that store the coefficient matrix in their top-left block:
$$ U = \begin{pmatrix} M/\psi & \cdot \\
\cdot & \cdot \end{pmatrix},$$
where $\psi \geq \| M \|$ is a normalization factor chosen to ensure that $U$ has operator norm at most~$1$.
Assuming access to QRAM, the QLSA of \cite{Blockencoding} has $\Tilde{\Ocal}_{d,\kappa_M, \frac{1}{\epsilon} }\left(\kappa_M \psi \right)$ complexity.

While QLSAs provide a quantum state proportional to the solution, 
it is not possible to extract the classical solution by a single measurement. 
Quantum Tomography Algorithms (QTAs) are needed to extract the classical solution. 
There are several papers improving QTAs, and the best QTA \cite{van2023quantum} has $\Ocal(\frac{d\varrho}{\epsilon})$  complexity, where $\varrho$ is a bound for the norm of the solution. 
The direct use of the QLSA from \cite{Blockencoding} and the QTA by \cite{van2023quantum} costs $\Tilde{\Ocal}_{d,\kappa_M, \frac{1}{\epsilon} }\left(d\kappa_M^2 \|M\|_F \frac{\|\sigma\|}{\epsilon}\right)$.
Mohammadisiahroudi et al. \cite{mohammadisiahroudi2023Accurately} used an iterative refinement approach employing limited precision QLSA+QTA with $\Tilde{\Ocal}_{d,\kappa_M, \frac{\|\sigma\|}{\epsilon} }\left(d\kappa_M \|M\|_F \right)$ complexity with $\Tilde{\Ocal}_{ \frac{\|\sigma\|}{\epsilon} }\left(d^2 \right)$ classical arithmetic operations.

\begin{theorem} \label{theo: QLSA solution}
 Using the linear solver of \cite{mohammadisiahroudi2023Accurately}, the \eqref{eq: modified normal equation} system can be built and solved to obtain a solution ${\widetilde{z}}^k$ satisfying $\|\rhat^k\|\leq\frac{\eta}{\sqrt{1+\theta}}\sqrt{\mu^k}$ with 
 with $\Tilde{\Ocal}_{m,\kappa_E^k, \frac{\|\hat{\sigma}^k\|}{\mu^k} }\left(m\kappa_E^k \|E^k\|_F \right)$ complexity and $\Tilde{\Ocal}_{ \frac{\|\hat{\sigma}^k\|}{\mu^k} }\left(mn \right)$ classical arithmetic operations, where $E^k=(D^k_{\Bhat})^{-1}\Ahat D^k$, and $\kappa_E^k$ is the condition number of $E^k$.
\end{theorem}
\begin{proof}
Building the \eqref{eq: modified normal equation} system in a classical computer requires matrix multiplications, which cost $\Ocal(m^2n)$ arithmetic operations.
We can write \eqref{eq: modified normal equation} as $E^k(E^k)^T \tilde{z}{}^k=\hat{\sigma}^k$, where $E^k=(D_{\Bhat}^{k})^{-1}\Ahat D^k$. 
Calculating $E^k$ and $\hat{\sigma}^k$ requires just $\Ocal(mn)$ arithmetic operations.
The authors in \cite{Blockencoding} proposed an efficient procedure to build and solve a linear system of the form $E^k(E^k)^T \tilde{z}{}^k=\hat{\sigma}^k$, with $\tilde{\Ocal}(\text{polylog}(\frac{n}{\epsilon})\kappa_E^k\|E^k\|_F)$ complexity.
Then, we can use the quantum linear system solver of \cite{mohammadisiahroudi2023Accurately} with 
$\Tilde{\Ocal}_{m,\kappa_E^k, \frac{\|\hat{\sigma}^k\|}{\mu^k} }\left(m\kappa_E^k \|E^k\|_F \right)$ complexity and $\Tilde{\Ocal}_{ \frac{\|\hat{\sigma}^k\|}{\mu^k} }\left(mn \right)$ classical arithmetic operations.
\end{proof}
In the next section, we apply the proposed modification of NES to develop an IF-QIPM.

\section{Inexact-Feasible Quantum 
 Interior Point Method}\label{sec: IF-QIPM}
Before developing the algorithm, first we define the central path as
\begin{equation*}
    \mathcal{CP}=\Big\{(x,y,s)\in\mathcal{PD}\ \big| \ x_is_i=\mu\  \text{ for }\ i\in\{1,\dots,n\}\Big\},
\end{equation*}
where $\mu=\frac{x^Ts}{n}$. For any $\theta\in[0,1)$, a small neighborhood of the central path is defined as 
\begin{align*}
    \mathcal{N}(\theta)=\Big\{(x,y,s)\in\mathcal{PD}^0\ \big| \|XSe-\mu e\|_2\leq \theta \mu \Big\}.
\end{align*}
Now, we develop the IF-QIPM using the modified NES with the short-step version of IPMs.
\begin{algorithm}[h]
\caption{Short-step IF-QIPM using QLSA} \label{alg: IF-QIPM}
\begin{algorithmic}[1]
 \STATE Choose $\zeta >0$, $\eta=0.1$, $\theta=0.7$and $\beta=(1-\frac{0.2}{\sqrt{n}})$. 
 \STATE $k \gets 0$
 \STATE Choose initial feasible interior solution $(x^0, y^0, s^0)\in \mathcal{N}(\theta)$ 
\WHILE {$(x^k,y^k,s^k)\notin \mathcal{PD}_\zeta$}
\STATE $\mu^k \gets \frac{(x^k)^Ts^k}{n}$
\STATE $\tilde{z}^k \gets$ Solve $\Mhat^k\tilde{z}^k=\hat{\sigma}^k$ using IR+QLSA+QTA of \cite{mohammadisiahroudi2023Accurately} with $\|\rhat^k\|\leq\frac{\eta}{\sqrt{1+\theta}}\sqrt{\mu^k}$.
    \STATE $\Delta y{}^k \gets ((D_{\Bhat}^{k})^{-1}A_{\Bhat}^{-1})^T\tilde{z}^k$
    \STATE $v^k\gets(v^k_{\Bhat},v^k_{\Nhat})=(D_{\Bhat}^{k}\rhat^k,0)$
    \STATE $\Delta s{}^k\gets c-A^Ty^k-s^k-A^T\Delta y{}^k$.
    \STATE $\Delta x{}^k\gets \beta_1\mu^k (S^{k})^{-1}e-x^k-(D^k)^2\Delta s{}^k-v^k$.
\STATE $(x^{k+1},y^{k+1},s^{k+1}) \gets (x^k,y^k,s^k)+(\Delta x^k,\Delta y^k,\Delta s^k)$
\STATE $k \gets k+1$
\ENDWHILE
\STATE Return $(x^k,y^k,s^k)$
\end{algorithmic}
\end{algorithm}
In the next section, we prove the convergence of the IF-QIPM and analyze its complexity.

\subsection{Convergence Analysis}
To prove the convergence of IF-QIPM, we use the analysis of IF-QIPM in \cite{IF-QIPMforLO}. The only difference is the choice of the Newton system: In \cite{IF-QIPMforLO}, the authors use OSS and in the current work we propose the use of MNES, however both compute $(\Delta x, \Delta y, \Delta s)$ such that
\begin{align*}
    A\Delta x&=0,\\
    A^T \Delta y +\Delta s&=0,\\
    X\Delta s+S\Delta x&=\beta_1 \mu e-X^ks^k+r',
\end{align*}
where $\|r'\|\leq \frac{\eta}{\sqrt{1+\theta}} \sqrt{\mu}$.

Next we provide relevant theory for our IF-QIPM in lemmas \ref{lemma: orthogonality and mu} and \ref{lemma: ramaining in the neighbor} and theorem \ref{theo: iteration bound}.  For the relevant proofs, we refer to lemmas 3.1 and 3.2 and theorem 2.6, respectively, in \cite{IF-QIPMforLO}.
\begin{lemma}\label{lemma: orthogonality and mu}
Let the step $(\Delta x,\Delta y, \Delta s)$ be calculated from \eqref{eq: modified normal equation} in each iteration of the IF-IPM. Then
\begin{align*}
    \Delta x^T\Delta s&=0,\\
    (x+\Delta x)^T(s+\Delta s)&\leq (\beta+\frac{\eta}{\sqrt{1+\theta}})x^Ts,\\
    (x+\Delta x)^T(s+\Delta s)&\geq (\beta-\frac{\eta}{\sqrt{1+\theta}})x^Ts. 
\end{align*} 
\end{lemma}
Now, we can show that the iterates of IF-QIPM remain in the neighborhood of the central path in Lemma~\ref{lemma: ramaining in the neighbor}, by using results of Lemma~\ref{lemma: orthogonality and mu}. 
\begin{lemma} \label{lemma: ramaining in the neighbor}
Let $(x^0,s^0,y^0)\in \mathcal{N}(\theta)$ for a given $\theta\in[0,1)$, then 
$(x^k,s^k,y^k)\in \mathcal{N}(\theta)$ for any $k\in \mathbb{N}$.
\end{lemma}

Based on Lemma~\ref{lemma: ramaining in the neighbor}, IF-IPM remains in the neighborhood of the central path, and it converges to the optimal solution if $\mu^k$ converges to zero. In Theorem~\ref{theo: iteration bound}, we prove that the algorithm reaches $\zeta$-optimal solution after $\Ocal(\sqrt{n}\log(\frac{{\mu}_0}{\zeta}))$  iteration.

\begin{theorem}\label{theo: iteration bound}
The sequence $\mu^k$ converges to zero linearly, and we have $\mu^k\leq \zeta$ after
$\Ocal(\sqrt{n}\log(\frac{{\mu}_0}{\zeta}))$ 
iterations.
\end{theorem}
%
%
This demonstrates that the IF-IPM achieves the best-known iteration complexity, and the proof holds for any values satisfying the following two conditions.
\begin{align}
    \beta&\leq(1-\frac{\eta+0.01}{\sqrt{n}}),\label{eq: param con1}\\
    \frac{\theta^2 +n(1-\beta)^2+\eta^2}{2^{3/2}(1-\theta)}+ \eta &\leq \theta(\beta-\frac{\eta}{\sqrt{n}}).\label{eq: param con2}
\end{align}
It is not hard to check that $\theta=0.7$ and $\eta=0.1$ satisfy these conditions.

\subsection{Complexity}
Let $L$ be the binary length of input data defined as 
$$L=mn+m+n+\sum_{i,j}\lceil\log(|a_{ij}|+1)\rceil+\sum_{i}\lceil\log(|c_{i}|+1)\rceil+\sum_{j}\lceil\log(|b_{j}|+1)\rceil.$$
An exact solution can be calculated by rounding \cite{Wright1997_Primal}, provided that we terminate with $\mu^k\leq2^{\Ocal(L)}.$
Accordingly, the IF-IPM may require $\Ocal(\sqrt{n}L)$ to determine an exact optimal solution; for more details see \cite[Chapter 3]{Wright1997_Primal}. The next theorem characterizes the total time complexity of the proposed IF-QIPM.
\begin{theorem}\label{theo: final complexity}
The proposed IF-QIPM of Algorithm~\ref{alg: IF-QIPM} determines a $\zeta$-optimal solution using at most $$\tilde{\Ocal}_{m,\kappa_{\Ahat}, \|\Ahat\|,\|\bhat\|,\mu^0, \frac{1}{\zeta}}\left(\sqrt{n}m^{1.5}\kappa_{\Ahat}\|\Ahat\|\frac{\omega^2}{\zeta^3}\right)$$
queries to the QRAM and $\tilde{\Ocal}_{\mu^0, \frac{1}{\zeta}}\left(n^{1.5}m\right)$ classical arithmetic operations.
\end{theorem}
\begin{proof}
The complexity analysis of the different parts of IF-QIPM \ref{alg: IF-QIPM} is outlined as follows:
\begin{itemize}
    \item After $\Ocal(\sqrt{n}\log(\frac{\mu^0}{\zeta}))$ the IF-QIPM obtains a $\zeta$-optimal solution. 
    \item In Theorem 4.2 of \citep{mohammadisiahroudi2022efficient}, the norm and condition number bounds of MNES are derived as $$\|\sigma^k\|=\Ocal\left(\frac{\|\Ahat\|+\|\bhat\|}{\zeta}\right), \quad \|E^k\|_F\leq\sqrt{m}\|E^k\|=\Ocal\left(\frac{\sqrt{m}}{\zeta}\|\Ahat\|\right),\quad\kappa^k_E=\Ocal\left(\frac{\omega^2}{\zeta^{2}}\kappa_{\Ahat}\right),$$
    where $\omega$ is a bound on $\|x^*,s^*\|_{\infty}$ for all $(x^*,y^*,s^*)\in \mathcal{PD}^*$.
    \item Applying Theorem~\ref{theo: QLSA solution}, the complexity of quantum subroutine to solve the MNES is 
    $$\tilde{\Ocal}_{m,\kappa_{\Ahat}, \|\Ahat\|,\|\bhat\|,\frac{1}{\zeta}}\left(m^{1.5}\kappa_{\Ahat}\|\Ahat\|\frac{\omega^2}{\zeta^3}\right).$$
    \item In each iteration of IF-QIPM, we need to build $E^k$ classically and load to QRAM with $\Ocal(mn)$ complexity. Also, some classical matrix products happen with $\Ocal(mn)$ cost. Thus, the classical cost per iteration is $\Ocal(mn)$. 
    \item Thus, in the worst case the IF-QIPM requires
$$\tilde{\Ocal}_{m,\kappa_{\Ahat}, \|\Ahat\|,\|\bhat\|,\mu^0, \frac{1}{\zeta}}\left(\sqrt{n}m^{1.5}\kappa_{\Ahat}\|\Ahat\|\frac{\omega^2}{\zeta^3}\right)$$
accesses to the QRAM and $\tilde{\Ocal}_{\mu^0, \frac{1}{\zeta}}\left(n^{1.5}m\right)$ classical arithmetic operations.
\end{itemize}
\end{proof}

\subsection{Improving the error dependence of the IF-QIPM}
To get an exact optimal solution, the time complexity contains the exponential term $2^L$. To address this problem, we can fix $\zeta=10^{-2}$ and improve the precision by iterative refinement in $\Ocal(L)$ iterations \cite{mohammadisiahroudi2022efficient}. The first iterative regiment method for linear optimization is proposed by Gleixner et al. \cite{iterative}. Here, we use
the iterative refinement of \cite{IF-QIPMforLO}, which is designed specifically for IF-QIPM as Algorithm\ref{alg:iterative refinement}. 
\begin{algorithm}[H]
\caption{IR-IF-QIPM of \cite{IF-QIPMforLO}} \label{alg:iterative refinement}
\begin{algorithmic}[1]
\REQUIRE \big($A \in \mathbb{R}^{m\times n}, b\in \mathbb{R}^{m}, c\in \mathbb{R}^{n},\zeta<\hat{\zeta}<1$ \big)
\STATE  $k \gets 1$ and $\nabla_0\gets1$
\STATE $({x}^1, {y}^1, {s}^1) \gets$ \textbf{solve} $(A,b,c)$ using IF-QIPM of Algorithm~\ref{alg: IF-QIPM} with $\hat{\zeta}$ precision 
\WHILE{$(x^k,y^k,s^k)\notin \mathcal{PD}_{\zeta}$}
\STATE $\nabla^k \gets \frac{1}{\big(x^k\big)^Ts^k}$
\STATE $(\hat{x}^k, \hat{y}^k, \hat{s}^k) \gets$ \textbf{solve} $(A,0,\nabla^ks^{k})$ using IF-QIPM of Algorithm~\ref{alg: IF-QIPM} with $\hat{\zeta}$ precision
\STATE $x^{k+1}\gets x^{k}+\frac{1}{\nabla^k} \hat{x}^k$ and $y^{k+1}\gets y^{k}+\frac{1}{\nabla^k}\hat{y}^k$
\ENDWHILE
\end{algorithmic}
\end{algorithm}
\begin{theorem}
    The total time complexity of finding exact optimal solution with IR-IF-QIPM is
$$\tilde{\Ocal}_{m,\kappa_{\Ahat}, \|\Ahat\|,\|\bhat\|,\mu^0}\left(\sqrt{n}m^{1.5}\kappa_{\Ahat}\|\Ahat\|\omega^2L\right)$$
with $\tilde{\Ocal}_{\mu^0}\left(n^{1.5}mL\right)$ classical arithmetic operation. 
\end{theorem}
\begin{proof}
    The proof is similar to the proof of \cite[Theorem 6.2]{IF-QIPMforLO}.
\end{proof}

In the next section, we investigate how preconditioning can help to mitigate the effect of the condition number with respect to $\kappa_{\Ahat}$ and $\omega$.
\section{Iteratively Refined IF-QIPM using preconditioned NES}\label{sec: condition}
To mitigate the impact of the condition number, we need to analyze how the matrices $M^k$ and $E^k$ evolve through the iterations.
As in Theorem 6.8 of \citep{Wright1997_Primal} or Lemma \MakeUppercase{\romannumeral 1}.42 of \citep{Roos2005_Interior}, considering the optimal partition $B$ and $N$, we have
\begin{equation}\label{eq:diagonal scaling}
    \frac{x_i^k}{s_i^k}=\Ocal\left(\frac{\Chat}{\mu^{k}}\right)\to\infty\text{  for }i\in B\qquad \text{and} \qquad\frac{x_i^k}{s_i^k}=\Ocal\left(\frac{\mu^{k}}{\Chat}\right)\to0\text{  for }i\in N,
\end{equation}
where $\Chat$ is a constant dependent on the LO problem's parameters. For a more detailed analysis, see pages 121-124 of \citep{Wright1997_Primal}.
To analyze the condition number of NES, we have 
$$M^k=A(D^k)^2A^T=A_B(D_B^k)^2A_B^T+A_N(D_N^k)^2A_N^T.$$
As the sequence of iterates converges to the optimal set, it is easy to see that $A_N(D_N^k)^2A_N^T\to 0$. Thus, the dominant component is the $A_B(D_B^k)^2A_B^T$ term.
If $A_B$ includes a basis of $A$, i.e. $\textup{rank}(A_B)=m$, then the condition number of $M^k$ will converge to a constant depending on $\frac{\max_{i\in B} x^2_i}{\min_{i\in B} x^2_i}$.
This implies that when the problem is primal nondegenerate, the condition number will converge to a constant, though that constant may be as big as $\Ocal(2^{2L})$ in the worst case. If $A_B$ has a rank less than $m$, then the condition number of $M^k$ goes to infinity with the rate of $\frac{1}{\mu^2}$, which can be addressed by iterative refinement.
%
%
In the proposed IF-QIPM, we initially choose basis $\Bbar$. 
If in each iteration of IF-QIPM, we choose basis $\Bhat^k$ as indices of $m$ largest $\frac{x_i^k}{s_i^k}$, then by modifying MNES, we can precondition it too. 
We refer the readers to \cite{monteiro2004uniform} for the algorithm for determining basis $\Bhat^k$.
Suppose that the LO problem is non-degenerate. As the trajectory generated by the IF-QIPM converges to the optimal solution, we have
\begin{align*}
    \Bhat^k & \to B\\
    A_N(D_N^k)^2A_N^T&\to 0\\
    M^k &\to A_B(D_B^k)^2A_B^T\\
    \Mhat^k=(D_{\Bhat}^{k})^{-1}A_{\Bhat}^{-1}M^k((D_{\Bhat}^{k})^{-1}A_{\Bhat}^{-1})^T&\to I.
\end{align*}
Thus, $(D_{\Bhat}^{k})^{-1}A_{\Bhat}^{-1}$ is a precondition for $M^k$. 
When the LO problem is degenerate, which is the more general setting, Theorem 2.2.3 of \cite{o2006use} asserts that the condition number of $\Mhat^k$ is bounded by $(\bar{\chi})^2$ where
\begin{equation}
    \bar{\chi} = \max\left\{ \|A_B^{-1} A\|_F :\ A_B {\rm ~is~a~basis~of~} A\right\}.
\end{equation}
Furthermore, based on Lemma 2.2.2 of \cite{o2006use}, we have $\|\Mhat^k\|_F=\Ocal(\bar{\chi})$.

To utilize this preconditioning method within our IF-QIPM framework, we adopt the following procedure.
\begin{enumerate}
    \item []\textbf{Step 1.} Choose basis $\Bhat^k$ as indices of $m$ largest, $\frac{x_i^k}{s_i^k}$ where $A_{\Bhat^k}$ are linearly independent
    \item []\textbf{Step 2.} Build block-encoding of $(D_{\Bhat}^{k})^{-1}$ and $A_{\Bhat}$
    \item []\textbf{Step 3.} Calculate $A_{\Bhat}^{-1}$ on the quantum computer
    \item []\textbf{Step 4.} Build $\Mhat^k=(D_{\Bhat}^{k})^{-1}A_{\Bhat}^{-1}M^k((D_{\Bhat}^{k})^{-1}A_{\Bhat}^{-1})^T$  on the quantum computer
    \item []\textbf{Step 5.} Find $\tilde{z}^k$ such that $\Mhat^k\tilde{z}^k=\hat{\sigma}^k+\rhat^k$ and $\|\rhat^k\|\leq\frac{\eta}{\sqrt{1+\theta}}\sqrt{\mu^k}$ on the quantum computer
    \item []\textbf{Step 6.} Calculate $\widetilde{\Delta y}{}^k=((D_{\Bhat}^{k})^{-1}A_{\Bhat}^{-1})^T\tilde{z}^k$.
    \item []\textbf{Step 7.} Calculate $v^k=(v^k_{\Bhat},v^k_{\Nhat})=(D_{\Bhat}^{k}\rhat^k,0)$.
    \item []\textbf{Step 8.} Calculate $\widetilde{\Delta s}{}^k=c-A^Ty^k-s^k-A^T\widetilde{\Delta y}{}^k$.
    \item []\textbf{Step 9.} Calculate $\widetilde{\Delta x}{}^k=\beta_1\mu^k (S^{k})^{-1}e-x^k-(D^k)^2\widetilde{\Delta s}{}^k-v^k$.
\end{enumerate}
It is straightforward to provide a convergence proof of an IF-QIPM that uses this procedure, since it produces feasible-inexact iterates satisfying  $\|\rhat^k\|\leq\frac{\eta}{\sqrt{1+\theta}}\sqrt{\mu^k}$. Thus, the iteration complexity of IF-QIPM using preconditioned NES is $\Ocal(\sqrt{n}\log(\frac{\mu^0}{\zeta}))$. 

In order to estimate the asymptotic scaling of the overall complexity, we begin by analyzing the cost of block-encoding the Newton system coefficient matrix in each iteration. 
\begin{proposition}
    Suppose $A$ and $A_{\Bhat}$ are stored in a QRAM data structure. Then, one can prepare a block-encoding of 
    $$ \Mhat^k=(D_{\Bhat}^{k})^{-1}A_{\Bhat}^{-1}M^k((D_{\Bhat}^{k})^{-1}A_{\Bhat}^{-1})^T$$
    using $\widetilde{\Ocal}_{m, n, \frac{1}{\epsilon}} \left(1\right)$ accesses to the QRAM and $ \widetilde{\Ocal}\left( n \right) $ arithmetic operations.
\end{proposition}

\begin{proof}
    First, observe that we always have classical access to $x^k$ and $s^k$. We can therefore store the nonzero entries of the matrices $\left( D^k \right)^2$ and $\left( D_{\Bhat}^k \right)^{-1}$ in QRAM using $\widetilde{\Ocal}_n (n)$ classical operations. From here, applying \cite[Lemma 50]{gilyen2019quantum} asserts that $\widetilde{\Ocal}_{n, \frac{1}{\xi_D}} (1)$ accesses to the QRAM suffices to construct an $\left( \alpha_{D^2}, \log(n) + 2, \xi_D \right)$-block-encoding of $(D_{\Bhat}^{k})^{2}$ and an $\left( \alpha_{D_{\Bhat}}, \log(n) + 2, \xi_D \right)$-block-encoding of $(D_{\Bhat}^{k})^{-1}$. Likewise, with $A$ and $A_{\Bhat}$ stored in QRAM, invoking \cite[Lemma 50]{gilyen2019quantum} we can construct a $\left( \left\| A \right\|_F , \log(n) + 2, \xi_A \right)$-block-encoding of $A$ and a $\left( \left\| A_{\Bhat} \right\|_F , \log(n) + 2, \xi_A \right)$-block-encoding of $A_{\Bhat}$, using $\widetilde{\Ocal}_{m, n, \frac{1}{\xi_A}} (1)$ accesses to the QRAM. 
    
    From here, we will analyze the cost of preparing block-encodings of the terms 
    \begin{align*}
        M^k&=A(D^k)^2A^T=A_B(D_B^k)^2A_B^T+A_N(D_N^k)^2A_N^T \\
        P  &= \left( D_{\Bhat}^{k}\right)^{-1}A_{\Bhat}^{-1}
    \end{align*}
    Having prepared block-encodings of $A$ and $\left( D^k \right)^2$, we can take their product 
    prepare an $\left(\alpha_{D^2} \left\| A \right\|^2_F , \Ocal \left(\log(n)\right), \xi_M \right)$-block-encoding of $M^k$ as
    $$U_A U_{D^2} U_{A^{\top}} = U_{M}.$$ 
     Similarly, having prepared a $\left( \left\| A_{\Bhat} \right\|_F , \log(n) + 2, \xi_A \right)$-block-encoding of $A_{\Bhat}$ for $A_{\Bhat}$, applying \cite[Corollary 3.4.13]{gilyen2019thesis}, we can prepare a $\left(  \left\| A_{\Bhat} \right\|_F \kappa_{A_{\Bhat}}, \log(n) + 2, \xi_A \right)$-block-encoding of $\left(A_{\Bhat}\right)^{-1}$. From here, applying \cite[Lemma 4]{Blockencoding}, we can take the product of block-encodings  of $\left(A_{\Bhat}\right)^{-1}$ and $(D_{\Bhat}^{k})^{-1}$, which yields an $\left(\alpha_{D^{-1}} \kappa_{A_{\Bhat}} \left\| A_{\Bhat} \right\|_F, \Ocal \left(\log(n)\right), \xi_P \right)$-block-encoding of $P$.

    Observe that we have prepared block-encodings $U_M$ and $U_P$ for $M^k$ and $P$ such that 
    $$ U_P U_M U_{P^{\top}}$$
    corresponds to a $\left(\alpha_{D^{-1}} \alpha^2_{D^2} \kappa_{A_{\Bhat}}^2 \left\| A \right\|^2_F \left\| A_{\Bhat} \right\|_F^2, \Ocal \left(\log(n)\right), \xi_P \right)$-block-encoding of $P$, since it is defined as the product of block-encodings \cite[Lemma 4]{Blockencoding}. The complexity result follows from observing that constructing the unitary 
    $$U_{\Mhat} = U_P U_M U_{P^{\top}} = U_P U_A U_{D^2} U_{A^{\top}} U_{P^{\top}}$$
    has a gate cost of $\widetilde{\Ocal}_{m,n,\frac{1}{\epsilon}}$, where we have properly set the parameters $\xi_A$ and $\xi_D$, such that $U_{\Mhat}$ implements $\Mhat$ up to error $\epsilon$. We also needed $\widetilde{\Ocal}_n (n)$ classical operations to create the QRAM data structures for $D^{-1}$ and $D^2$. The proof is complete. 
\end{proof}

\begin{corollary}\label{corr:QLS-NES}
    Suppose $A$, $A_{\Bhat}$, $x^k$, $s^k$ and $\hat{\sigma}^k$ are stored in a QRAM data structure and define 
    $$\alpha := \alpha_{D^{-1}} \alpha^2_{D^2}.$$
    Then, one can obtain a $\frac{\eta}{\sqrt{1+\theta}} \sqrt{\mu^k}$-precise solution (in $\ell_{\infty}$-norm) to the linear system 
    $$\Mhat^k\tilde{z}^k=\hat{\sigma}^k,$$
    using at most 
    $$ \widetilde{\Ocal}_{m, n, \frac{1}{\epsilon}} \left( m \alpha \kappa_{A_{\Bhat}}^2   \kappa_{\Mhat} \left\| A \right\|^2_F \left\| A_{\Bhat} \right\|_F^2 \right) $$
    QRAM accesses and $ \widetilde{\Ocal} \left( mn \right) $ arithmetic operations. 
\end{corollary}

\begin{proof}
    Using the linear systems algorithm from \cite{mohammadisiahroudi2023Accurately} with the subnormalization factor $\alpha_{D^{-1}} \alpha^2_{D^2} \kappa_{A_{\Bhat}}^2 \left\| A \right\|^2_F \left\| A_{\Bhat} \right\|_F^2$ and condition number $\kappa_{\Mhat}$ gives the result. 
\end{proof}

We are now in a position to state the complexity of the Iteratively Refined IF-QIPM using preconditioned NES. 

\begin{theorem}
    Suppose that the LO problem data $A,b,c$ is stored in QRAM. Then, the Iteratively Refined IF-QIPM using preconditioned NES obtains an $\epsilon$-precise solution to the primal-dual LO pair $(P)-(D)$ using at most 
    $$ \widetilde{\Ocal}_{m, n, \frac{1}{\epsilon}} \left( \sqrt{n} m \alpha \kappa_{A_{\Bhat}}^2  \kappa_{\Mhat} \left\| A \right\|^2_F  \left\| A_{\Bhat} \right\|_F^2 \right) $$
    QRAM accesses and $ \widetilde{\Ocal}_{\frac{1}{\epsilon}} \left( mn^{1.5} \right) $ arithmetic operations, where $\alpha \geq \| \Mhat^k \|_F$ for all $k$. 
\end{theorem}

\begin{proof}
    The result follows upon adjusting the proof of Theorem \ref{theo: final complexity} to account for the result in Corollary \ref{corr:QLS-NES}.
\end{proof}
Based on analysis of \cite{o2006use}, we have $\alpha=\Ocal(\|\Mhat\|)=\Ocal(\bar{\chi})$ and $\kappa_{\Mhat}=\Ocal(\bar{\chi}^2)$. As $\Bhat$ always forms a basis for $A$. it is reasonable to assume that $\|A_{\Bhat}\|_F=\Ocal(\|A_{\Bhat}\|)$ and $\kappa_{A_{\Bhat}}=\Ocal(\kappa_{A})$. Thus, we can simplify the quantum complexity to 
$$ \widetilde{\Ocal}_{m, n, \frac{1}{\epsilon}} \left( \sqrt{n} m \bar{\chi}^3 \kappa_{A}^2   \left\| A \right\|^4_F   \right) $$
with $ \widetilde{\Ocal}_{\frac{1}{\epsilon}} \left( mn^{1.5} \right) $ arithmetic operations.

\section{Numerical Experiments}\label{sec: numerical}
In this section, we present a series of numerical experiments aimed at elucidating the behavior of various QIPMs. Additionally, we compare the impact of employing iterative refinement versus preconditioning techniques. It is important to note that our analysis presupposes access to QRAM; however, it is essential to highlight that physical QRAM infrastructure has yet to be realized. Likewise, QLSAs remain beyond the capabilities of existing quantum hardware.

Consequently, our experiments are conducted using the IBM Qiskit HHL simulator, and it is imperative to acknowledge that our numerical results cannot be extrapolated to gauge the performance of QIPMs and QLSAs on actual quantum hardware. Simulating quantum computers on classical computers is known to be exponentially time-consuming, which precludes any empirical time comparison between classical and quantum methodologies at this juncture. Notwithstanding, we are capable of simulating QLSAs for problems with limited numbers of variables and manageable condition numbers.

Our primary focus therefore centers on presenting numerical findings pertaining to QIPMs, and we refrain from presenting QLSA results in this paper. Interested readers are referred to \cite{mohammadisiahroudi2023Accurately} for comprehensive numerical experiments related to QLSAs. 
Each of the algorithms discussed in this paper have been implemented in Python and are readily accessible on our GitHub repository at \url{https://github.com/QCOL-LU}. Our Python package encompasses a versatile array of QIPMs designed to solve linear, semidefinite, and second-order cone optimization problems. To enhancing their versatility and efficacy, we have also incorporated iterative refinement techniques into both Quantum Linear System Algorithms (QLSAs) and QIPMs. Users are offered the flexibility to conduct experiments with QIPMs using either classical or quantum linear solvers, with the option to employ preconditioning.

For our experimental setup, we employ the LOP generators described in \cite{mohammadisiahroudi2023generating}. These generators have been demonstrated to produce randomly generated Linear Optimization Problems (LOPs) with predefined optimal and interior solutions, thereby facilitating the evaluation of IF-QIPMs. Furthermore, these generators offer users the flexibility to control various characteristics of the problems, including the condition number of the coefficient matrix—a critical parameter for assessing QIPMs' performance.
Our numerical experiments were conducted on a workstation equipped with Dual Intel Xeon{\textregistered} CPU E5-2630~@ 2.20 GHz, featuring 20 cores and 64 GB of RAM.

\begin{figure}[tbhp] \centering
\subfloat[A nondegenerate LO with $\kappa_A=10$]{\label{fig: WANDCon}\includegraphics[width =0.5 \linewidth, trim = {.1cm .1cm .1cm .1cm}, clip]{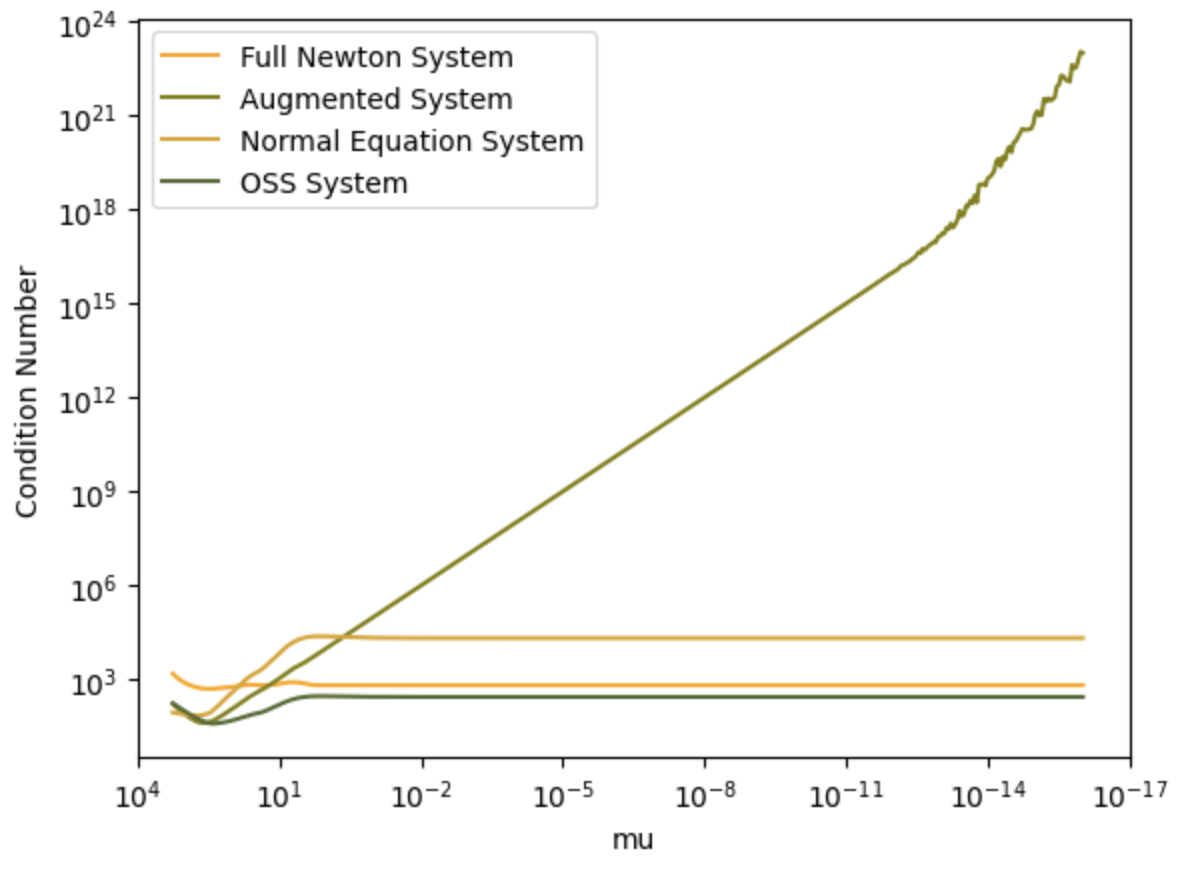}}
\subfloat[A nondegenerate LO with $\kappa_A=10^{6}$]{\label{fig: IANDCon}\includegraphics[width =0.5 \linewidth, trim = {.1cm .1cm .1cm .1cm}, clip]{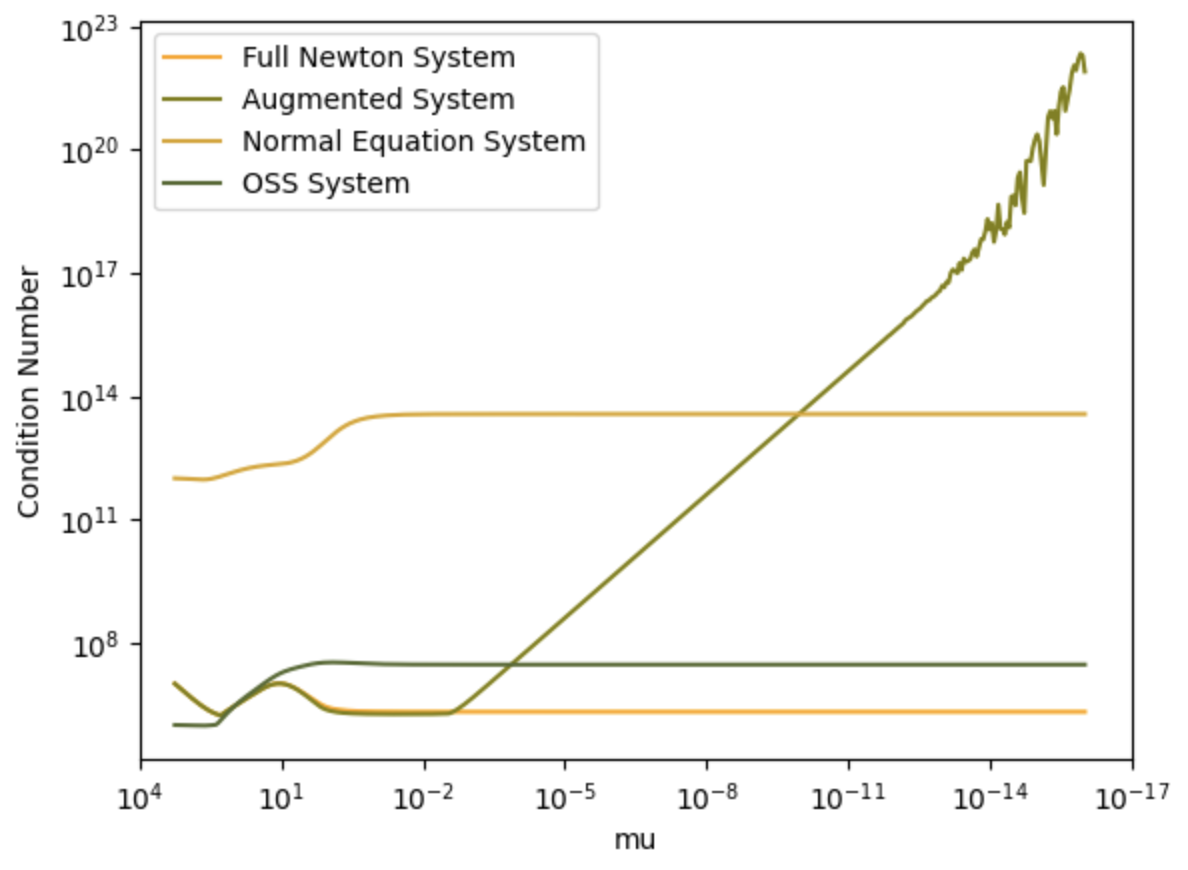}}\\
\subfloat[A degenerate LOP with $\kappa_A=10$]{\label{fig: WADCon}\includegraphics[width =0.5 \linewidth, trim = {.1cm .3cm .1cm .1cm}, clip]{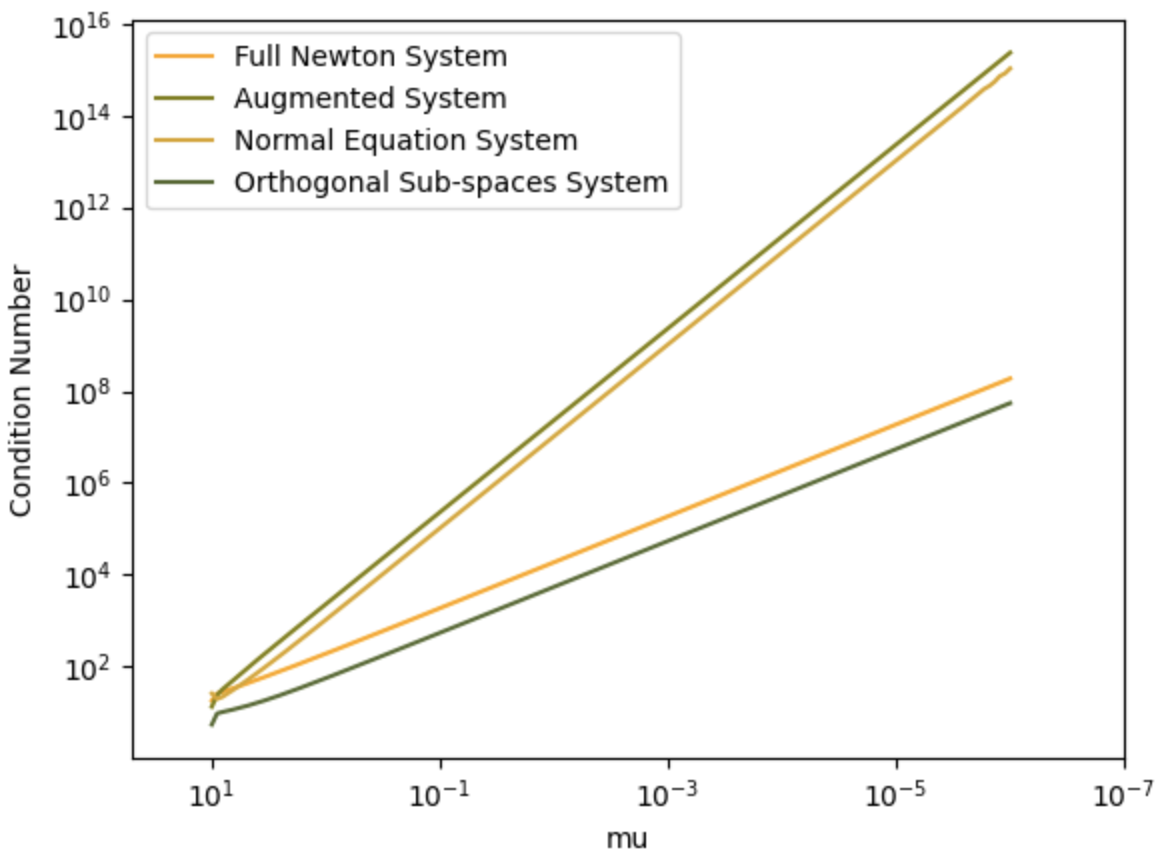}}
\subfloat[A degenerate LOP with $\kappa_A=10^{6}$]{\label{fig: IADCon}\includegraphics[width =0.5 \linewidth, trim = {.1cm .1cm .1cm .1cm}, clip]{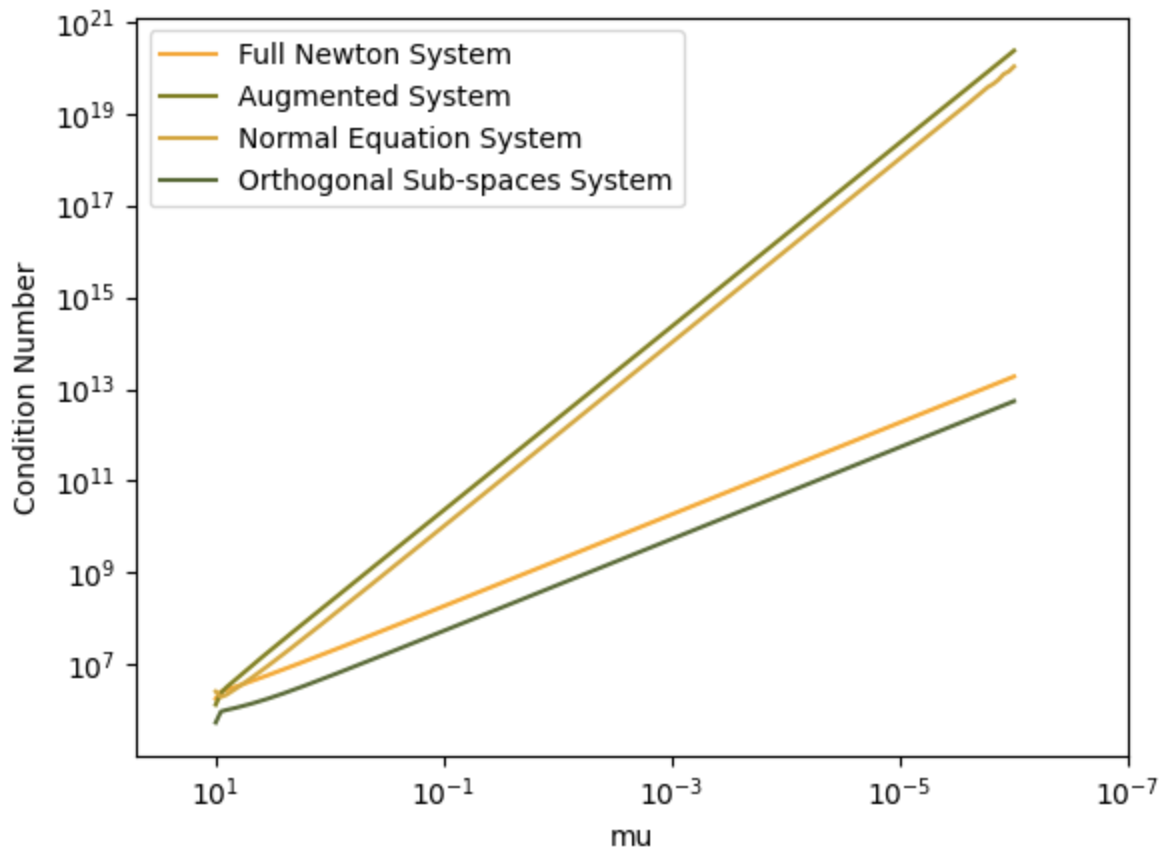}}
\caption{Condition number trend of different Newton systems for different types of LOPs. } \label{fig: con-sys}
\end{figure}
To illustrate how the condition numbers of different Newton systems evolve in IPMs, Fig.~\ref{fig: con-sys} shows the condition number of different linear systems trend for four problems. As Fig.~\ref{fig: WANDCon} shows, the condition number of FNS, OSS, and NES converge to a constant for nondegenerate LOPs with a well-conditioned matrix $A$. However, the condition number of the augmented system may go to infinity,  even for nondegenerate well-conditioned problems, as approaching to the unique optimal solution. For nondegenerate problems with ill-conditioned matrices, Fig.~\ref{fig: IANDCon}, the condition number of NES, OSS, and FNS still converge to a constant which can be very large, like $10^{12}$. If the LO problem is degenerate, Fig.~\ref{fig: WADCon}, the condition number of all Newton systems goes to infinity as approaching the optimal solution. However, FNS and OSS have a better rate than NES and AS. The worst case happens when the problem is degenerate and matrix $A$ is ill-conditioned, Fig.~\ref{fig: IADCon}. In this case, the condition number of NES can be as large as $10^{20}$ for $\mu=10^{-6}$. As these figures illustrate, the condition number of the Newton systems is affected by the condition number of matrix $A$ and the degeneracy status of the problem. Generally, OSS has a better condition number than the NES. In the next figures, we show how iterative refinement and preconditioning can mitigate the condition number of Newton systems, especially for NES.

\begin{figure}[tbhp] \centering
\subfloat[Without iterative refinement]{\label{fig: IANDQIPM}\includegraphics[width =0.5 \linewidth, trim = {.1cm .1cm .1cm .1cm}, clip]{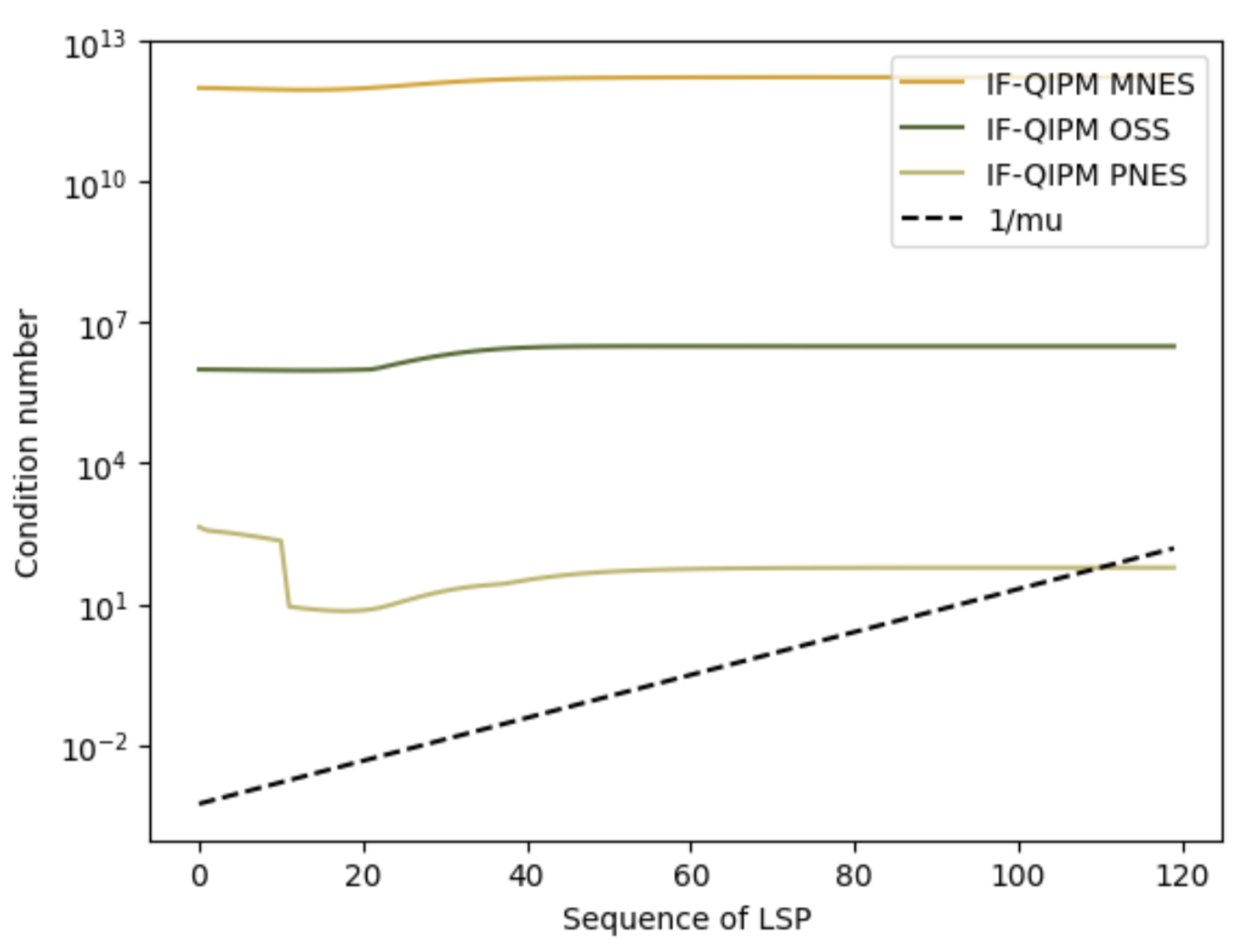}}
\subfloat[With iterative refinement]{\label{fig: IANDIR}\includegraphics[width =0.5 \linewidth, trim = {.1cm .1cm .1cm .1cm}, clip]{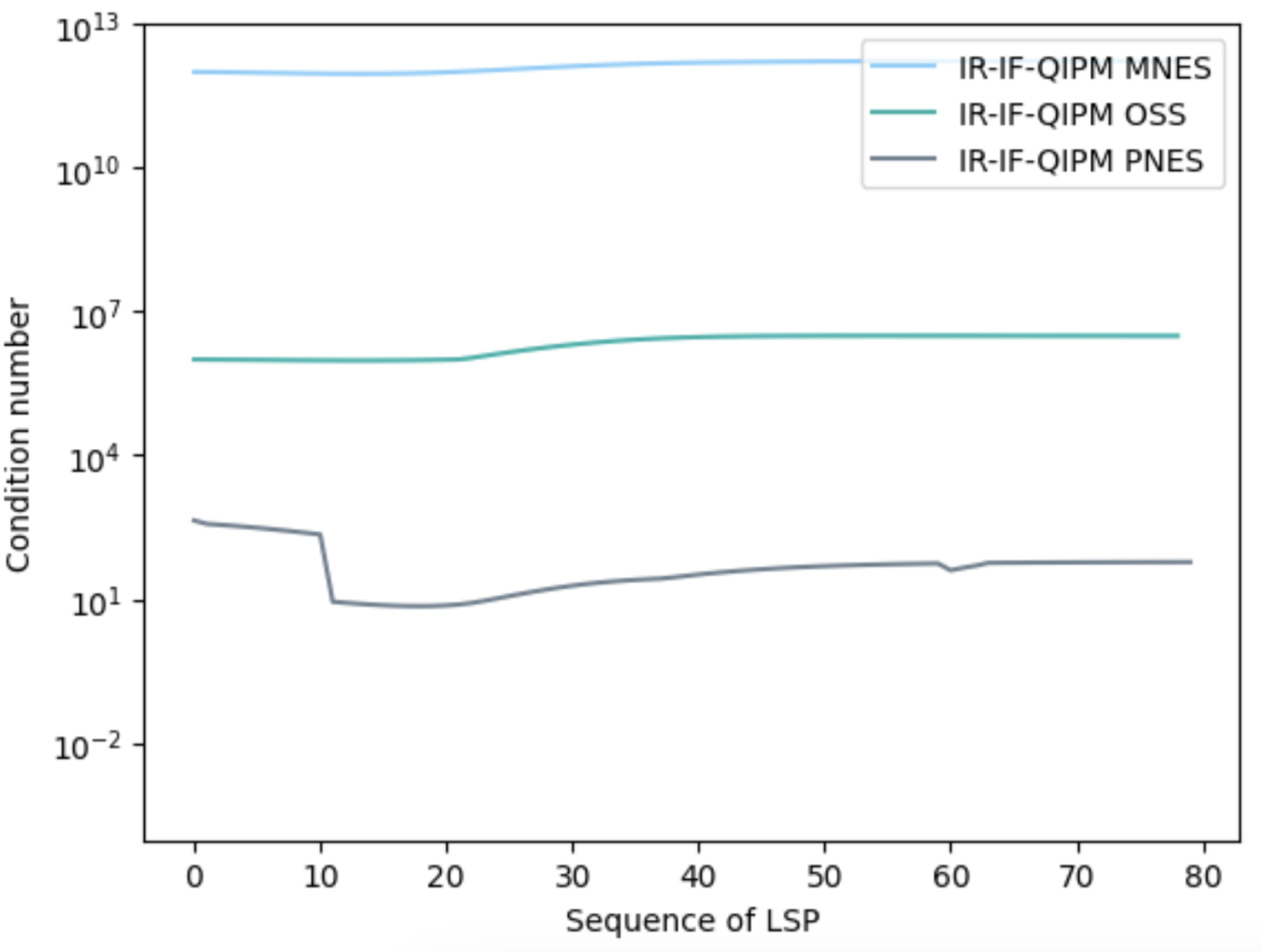}}
\caption{Condition number sequence of linear systems arising in different IF-QIPM to solve a nondegenerate LOP with
 $\kappa_A=10^{6}$ } \label{fig: IAND}
\end{figure}
Fig.~\ref{fig: IAND} shows the performance of different IF-QIPMs with respect to condition number to solve a nondegenerate problem with an ill-conditioned matrix $A$. As we can see, Preconditioned NES (PNES) has a significantly smaller condition number, even better than OSS. However, iterative refinement, Fig.~\ref{fig: IANDIR}, is not helping with the condition number since for this type of problem, early stopping the IF-QIPM and restarting it will not change the condition number as it is almost  constant, dependent on the condition number of $A$.
\begin{figure}[tbhp] \centering
\subfloat[Without iterative refinement]{\label{fig: WADQIPM}\includegraphics[width =0.5 \linewidth, trim = {.1cm .1cm .1cm .1cm}, clip]{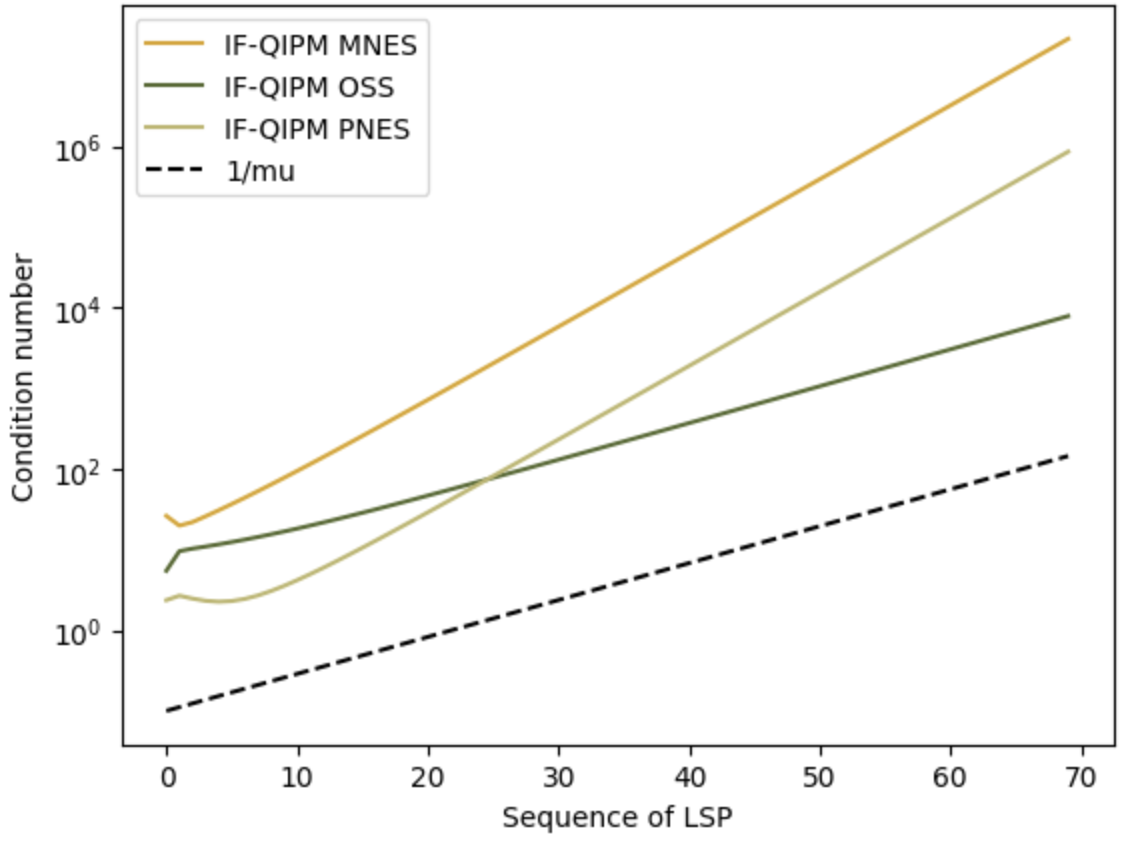}}
\subfloat[With iterative refinement]{\label{fig: WADIR}\includegraphics[width =0.5 \linewidth, trim = {.1cm .1cm .1cm .1cm}, clip]{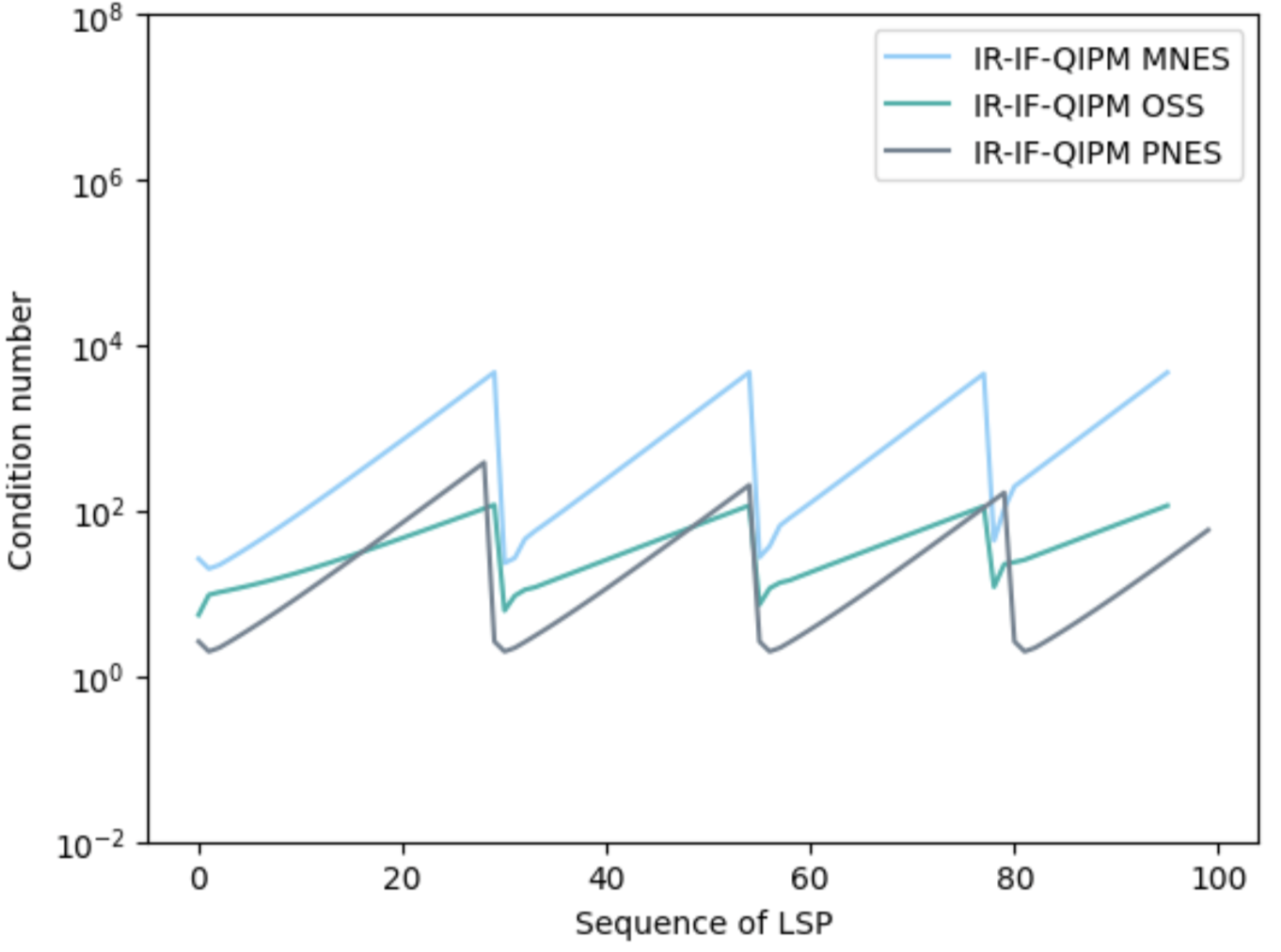}}
\caption{Condition number sequence of linear systems arising in different IF-QIPM to solve a degenerate LOP with
 $\kappa_A=10$ } \label{fig: WAD}
\end{figure}
The performance of IF-QIPMs for solving a primal degenerate LOP with well-conditioned coefficient matrix $A$ is depicted in Fig.~\ref{fig: WAD}. Although for PNES, there is a theoretical condition number bound, however as in this problem, depending on the input data this bound can be exponentially large. The condition number can grow with a slightly lower rate, but at the same rate as the one for for MNES. On the other hand, for degenerate problems, iterative refinement can help with condition numbers. As Fig.~\ref{fig: WADIR} shows, in the iterative refinements steps, we stop IF-QIPMs early, when $\mu=10^{-2}$ and so the condition number remains bounded. Then we restart the IF-QIPM for the refining problems, where the condition number is as low as the initial condition number. By IR, the condition number will not grow above an upper bound.

\begin{figure}[tbhp] \centering
\subfloat[Without iterative refinement]{\label{fig: IADQIPM}\includegraphics[width =0.5 \linewidth, trim = {.1cm .1cm .1cm .1cm}, clip]{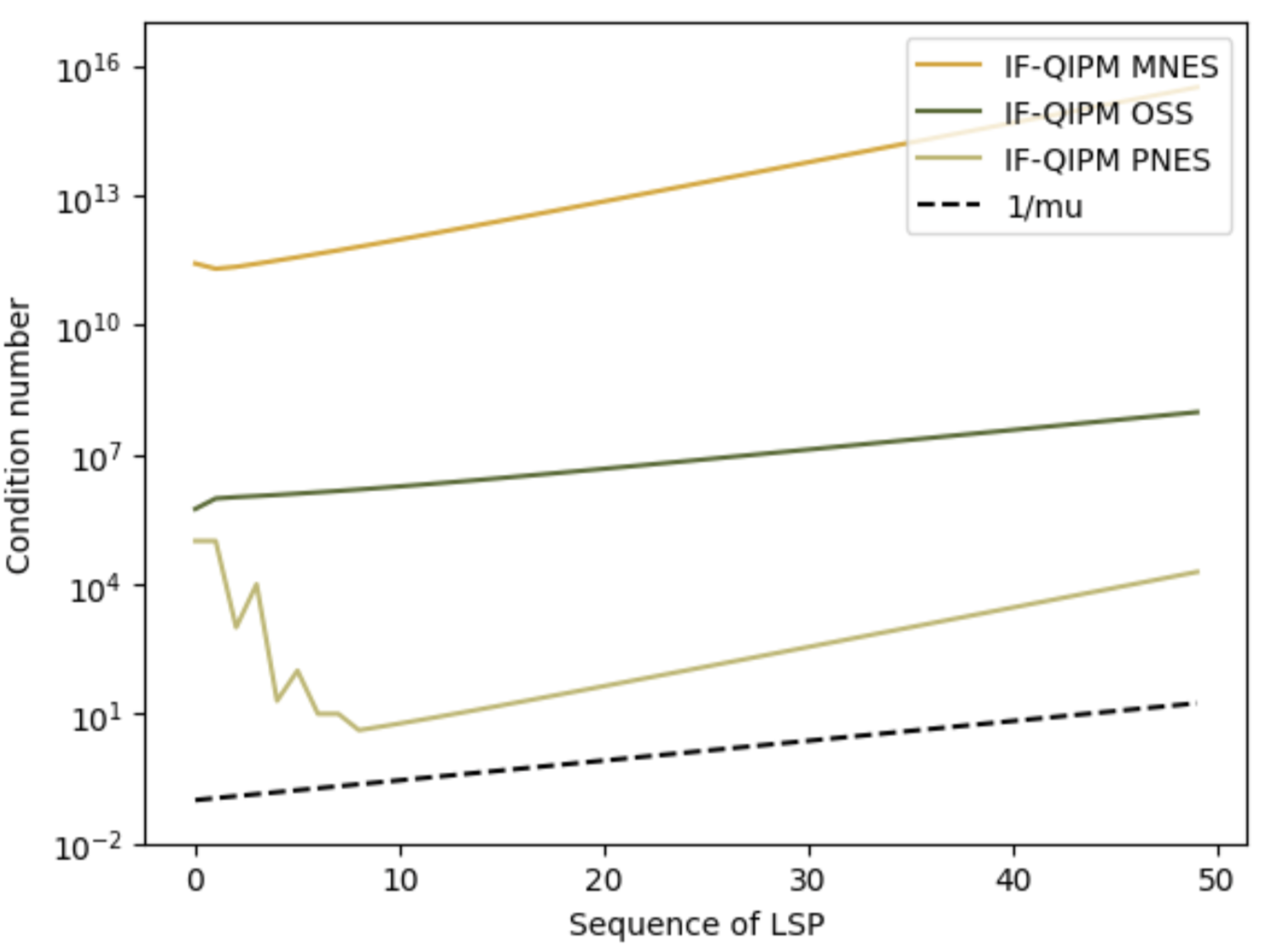}}
\subfloat[With iterative refinement]{\label{fig: IADIR}\includegraphics[width =0.5 \linewidth, trim = {.1cm .1cm .1cm .1cm}, clip]{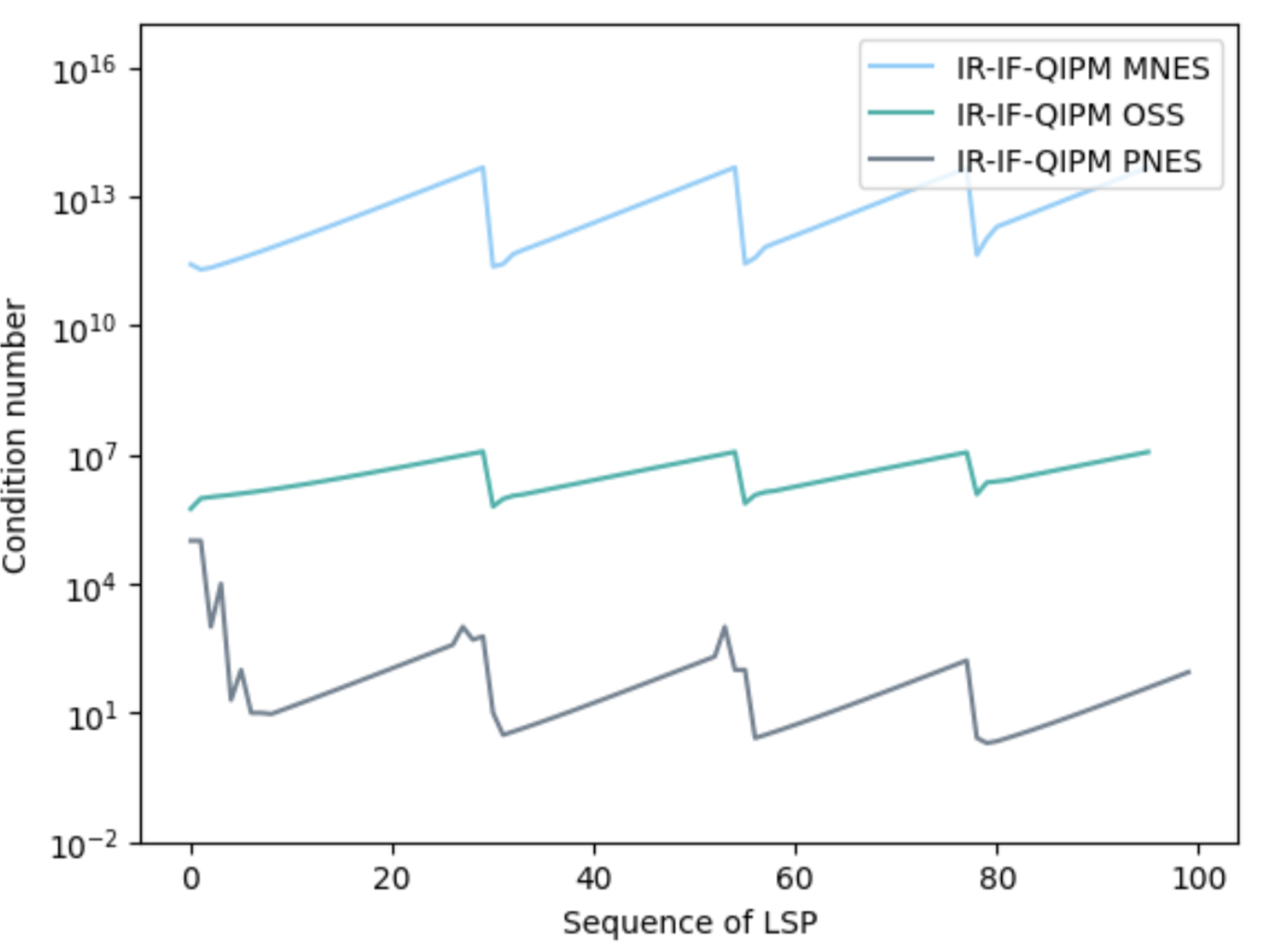}}
\caption{Condition number sequence of linear systems arising in different IF-QIPM to solve a degenerate LOP with
 $\kappa_A=10^{6}$ } \label{fig: IAD}
\end{figure}
Fig.~\ref{fig: IAD} shows how iterative refinement coupled with preconditioning can keep the condition number of the NES bounded during iterations of the QIPM for this challenging degenerate LOP with an ill-conditioned matrix. All in all, iterative refinement can mitigate the impact of degeneracy on the condition number of Newton systems. On the other hand, preconditioning is effective for addressing problems with an ill-conditioned matrix $A$.

We also solved 100 randomly generated problems with different IF-QIPMs using the IBM QISKIT simulator. Fig.~\ref{fig:stat} shows some statistics of solved problems. As we can see, with thwe OSS system we could not solve problems with more than 8 variables as the size of the linear systems that could be solved by the Qiskit simulator is limited to 16. In addition, the OSS has a nonsymmetric $n$-by-$n$ coefficient matrix. However, with MNES, we could solve an LOP with a million variables and 16 constraints as the dimension of MNES is dependent on the number of constraints. These results show that the proposed IF-QIPM is more adaptable to near-term devices. In addition, we can see that the iterative refinement coupled with preconditioning enables the solution of the problem with a larger condition number. In addition, with both inner and outer iterative refinement, we could improve the precision from $10^{-1}$ to $10^{-4}$ on average.
\begin{figure}
    \centering
    \includegraphics[scale=0.46]{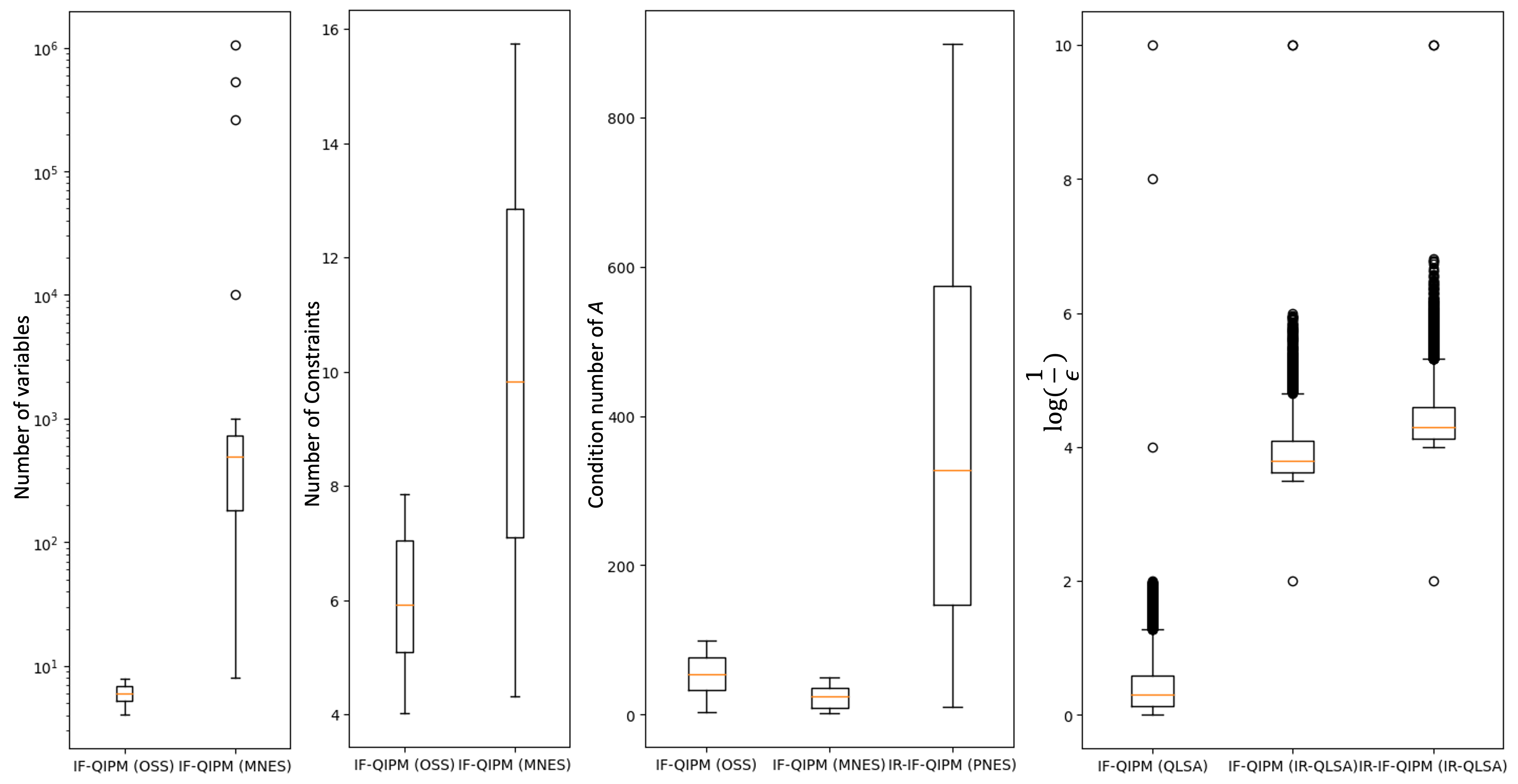}
    \caption{Statistics of 100 randomly generated LOPs solved by IF-QIPM using QISKIT Simulator (max time =2 hrs) }
    \label{fig:stat}
\end{figure}
\section{Conclusion}\label{sec: conclusion}
In this paper, we propose an inexact-feasible quantum interior point method in which we solve a modified normal equation system with QLSA+QTA. In addition, we apply an iterative refinement and preconditioning to mitigate the effect of condition number on the complexity of QIPMs. These classical ideas lead to some improvements in QIPMs outlines as follows:
\begin{itemize}
    \item By modifying NES, in each iteration of the proposed IF-QIPM, we solve a linear system with $m$-by-$m$ symmetric positive definite matrix, which is smaller than OSS with $n$-by$n$ nonsymmetric matrix. In other words, the proposed IF-QIPM needs fewer Qubits and gates. 
    \item We use an iterative refinement scheme coupled with preconditioning the NES that builds a uniform bound on the condition number and speeds up QIPMs w.r.t. precision, and condition number.
    \item By preconditioning the NES in the quantum setting, we achieved speed up w.r.t. the dimension compared to classical inexact approaches.
\end{itemize}
In Table~\ref{tab:compelxities}, the complexities of some recent classical and quantum IPMs are provided. As we can see, IR-IF-QIPM with MNES achieves the best complexity of IR-IF-QIPM using OSS with slightly better dependence on $\|A\|$, due to using quantum solver of \cite{mohammadisiahroudi2023Accurately}. By switching to preconditioned NES, complexity gets better with respect to a dimension but higher dependence on $\|A\|_F$. It is natural that calculating preconditioner on a quantum machine will have better complexity with respect to dimension but the challenge is addressing normalization factors in block-encoding. Another difference is dependence on $\bar{\phi}$ instead of $\omega$ which is an upper-bound for norm of optimal solution. It should be mentioned that both $\omega$ and $\bar{\phi}$ are constants depending on input data. However, for some problems, they can be extremely large. On the other hand don't condition number bound and numerical results, the one advantage of iterative refinement and preconditioning is mitigating the condition number. Mostly iterative refinement avoids the growing condition number of the Newton system in degenerate problems, and preconditioning alleviates the impact of the condition number of matrix $A$. All in all, QIPMs have the potential to speed up the solution of LOPs with respect to dimension compared to classical but they are more dependent on condition number and norm of the coefficient matrix. In this paper, we explored some classical ideas like using a better formulation of Newton's system and using iterative refinement coupled with preconditioning to shorten this gap. 
\begin{table}[]
\centering
\resizebox{\textwidth}{!}{%
\begin{tabular}{|c|c|c|c|c|c|}
\hline
Algorithm                & System & Linear System Solver & Quantum Complexity & Classical Complexity & Bound for $\kappa$ \\ \hline
IPM with Partial Updates \cite{Roos2005_Interior} & NES           &            Low rank updates          &                    &        $\Ocal(n^{3}L)$              &          \\ \hline
Feasible IPM \cite{Roos2005_Interior}            & NES           & Cholesky             &                    &                $\Ocal(n^{3.5}L)$      &          \\ \hline
II-IPM   \cite{Monteiro2003_Convergence}                & PNES           & PCG                   &                    &                $\Ocal(n^{5}L\bar{\chi}^2)$      &         $\bar{\chi}^2$  \\ \hline
II-QIPM \cite{mohammadisiahroudi2022efficient}                 & NES           & QLSA+QTA             &          $\tilde{\Ocal}_{n,\kappa_{A}, \omega}(n^{4}L\kappa_A^{2}\omega^{4}\|A\|^{4})$          &  $\tilde{\Ocal}_{ \omega}(n^{4}L)$                    &   $\Ocal(\kappa_{\Ahat}^{2} \omega^4)$       \\ \hline
IF-QIPM    \cite{IF-QIPMforLO}              & OSS           & QLSA+QTA             &    $\tilde{\Ocal}_{n,\kappa_{A}, \omega}(n^{2}L\kappa_A\omega^{2}\|A\|^{2})$                &       $\tilde{\Ocal}_{\mu^0}(n^{2.5}L)$              &     $\Ocal(\kappa_{\Ahat} \omega^2)$     \\ \hline
The proposed IR-IF-IPM                & MNES          & CG             &                    &    $\tilde{\Ocal}_{\mu^0}(n^{2.5}L\kappa_A^2\omega^4)$   &      $\Ocal(\kappa_{\Ahat}^{2} \omega^4)$    \\ \hline
The proposed IR-IF-IPM                & PNES          & PCG                  &                    &    $\tilde{\Ocal}_{\mu^0}(n^{3.5}L\bar{\chi}^2)$    &         $\bar{\chi}^2$ \\ \hline
The proposed IR-IF-QIPM                & MNES          & QLSA+QTA             &             $\tilde{\Ocal}_{n,\kappa_{\Ahat}, \|\Ahat\|,\|\bhat\|,\mu^0}(n^{2}L\kappa_{\Ahat}\omega^2\|\Ahat\|)$       &          $\tilde{\Ocal}_{\mu^0}(n^{2.5}L)$            &     $\Ocal(\kappa_{\Ahat}^{2} \omega^4)$     \\ \hline
The proposed IR-IF-QIPM               & PNES          & QLSA+QTA             &      $ \widetilde{\Ocal}_{ n, \frac{1}{\epsilon}} ( n^{1.5} L\bar{\chi}^3 \kappa_{A}^2   \left\| A \right\|^4_F   ) $              &            $\tilde{\Ocal}_{\mu^0}(n^{2.5})$          &     $\bar{\chi}^2$     \\ \hline
\end{tabular}%
}
\caption{Complexity of different IPMs for LO}
\label{tab:compelxities}
\end{table}

Using MNES also enables regularizing the Newton system. It is worth exploring the regularization in the quantum setting to address the impact of condition number on QIPMs. In addition, the proposed IR-IF-QIPM with preconditioned NES can be generalized to other conic problems such as Semi-definite optimization where the size of Newton systems may grow quadratically for large-scale problems. 

\begin{acks}
This work is supported by Defense Advanced Research Projects Agency as part of the project W911NF2010022: {\em The Quantum
Computing Revolution and Optimization: Challenges and Opportunities}. This work is also supported by National Science Foundation CAREER DMS-2143915.
\end{acks}

\bibliographystyle{ACM-Reference-Format}
\bibliography{sample-base}


\begin{thebibliography}{30}


\ifx \showCODEN    \undefined \def \showCODEN     #1{\unskip}     \fi
\ifx \showDOI      \undefined \def \showDOI       #1{#1}\fi
\ifx \showISBNx    \undefined \def \showISBNx     #1{\unskip}     \fi
\ifx \showISBNxiii \undefined \def \showISBNxiii  #1{\unskip}     \fi
\ifx \showISSN     \undefined \def \showISSN      #1{\unskip}     \fi
\ifx \showLCCN     \undefined \def \showLCCN      #1{\unskip}     \fi
\ifx \shownote     \undefined \def \shownote      #1{#1}          \fi
\ifx \showarticletitle \undefined \def \showarticletitle #1{#1}   \fi
\ifx \showURL      \undefined \def \showURL       {\relax}        \fi
\providecommand\bibfield[2]{#2}
\providecommand\bibinfo[2]{#2}
\providecommand\natexlab[1]{#1}
\providecommand\showeprint[2][]{arXiv:#2}

\bibitem[Al-Jeiroudi and Gondzio(2009)]%
        {Gondzio2009_Convergence}
\bibfield{author}{\bibinfo{person}{Ghussoun Al-Jeiroudi} {and} \bibinfo{person}{Jacek Gondzio}.} \bibinfo{year}{2009}\natexlab{}.
\newblock \showarticletitle{Convergence analysis of the inexact infeasible interior-point method for linear optimization}.
\newblock \bibinfo{journal}{\emph{Journal of Optimization Theory and Applications}} \bibinfo{volume}{141}, \bibinfo{number}{2} (\bibinfo{year}{2009}), \bibinfo{pages}{231--247}.
\newblock
\urldef\tempurl%
\url{https://doi.org/10.1007/s10957-008-9500-5}
\showDOI{\tempurl}


\bibitem[Ambainis(2012)]%
        {ambainis2012variable}
\bibfield{author}{\bibinfo{person}{Andris Ambainis}.} \bibinfo{year}{2012}\natexlab{}.
\newblock \showarticletitle{Variable time amplitude amplification and quantum algorithms for linear algebra problems}. In \bibinfo{booktitle}{\emph{STACS'12 (29th Symposium on Theoretical Aspects of Computer Science)}}, Vol.~\bibinfo{volume}{14}. LIPIcs, \bibinfo{pages}{636--647}.
\newblock


\bibitem[Augustino et~al\mbox{.}(2023)]%
        {Augustino2021_quantum}
\bibfield{author}{\bibinfo{person}{Brandon Augustino}, \bibinfo{person}{Giacomo Nannicini}, \bibinfo{person}{Tam{\'a}s Terlaky}, {and} \bibinfo{person}{Luis~F. Zuluaga}.} \bibinfo{year}{2023}\natexlab{}.
\newblock \showarticletitle{A Quantum Interior Point Method for Semidefinite Optimization Problems}.
\newblock \bibinfo{journal}{\emph{Quantum}}  \bibinfo{volume}{7} (\bibinfo{year}{2023}), \bibinfo{pages}{1110}.
\newblock


\bibitem[Augustino et~al\mbox{.}(2021)]%
        {augustino2021inexact}
\bibfield{author}{\bibinfo{person}{Brandon Augustino}, \bibinfo{person}{Tam{\'a}s Terlaky}, \bibinfo{person}{Mohammadhossein Mohammadisiahroudi}, {and} \bibinfo{person}{Luis~F Zuluaga}.} \bibinfo{year}{2021}\natexlab{}.
\newblock \showarticletitle{An inexact-feasible quantum interior point method for second-order cone optimization}.
\newblock \bibinfo{journal}{\emph{Tech Report}} (\bibinfo{year}{2021}).
\newblock


\bibitem[Casares and Martin-Delgado(2020)]%
        {Casares2020_quantum}
\bibfield{author}{\bibinfo{person}{PAM Casares} {and} \bibinfo{person}{MA Martin-Delgado}.} \bibinfo{year}{2020}\natexlab{}.
\newblock \showarticletitle{A quantum interior-point predictor--corrector algorithm for linear programming}.
\newblock \bibinfo{journal}{\emph{Journal of Physics A: Mathematical and Theoretical}} \bibinfo{volume}{53}, \bibinfo{number}{44} (\bibinfo{year}{2020}), \bibinfo{pages}{445305}.
\newblock
\urldef\tempurl%
\url{https://doi.org/10.1088/1751-8121/abb439}
\showDOI{\tempurl}


\bibitem[Chakraborty et~al\mbox{.}(2018)]%
        {Blockencoding}
\bibfield{author}{\bibinfo{person}{Shantanav Chakraborty}, \bibinfo{person}{Andr{\'a}s Gily{\'e}n}, {and} \bibinfo{person}{Stacey Jeffery}.} \bibinfo{year}{2018}\natexlab{}.
\newblock \showarticletitle{The power of block-encoded matrix powers: improved regression techniques via faster Hamiltonian simulation}.
\newblock \bibinfo{journal}{\emph{arXiv preprint arXiv:1804.01973}} (\bibinfo{year}{2018}).
\newblock


\bibitem[Childs et~al\mbox{.}(2017)]%
        {childs2017quantum}
\bibfield{author}{\bibinfo{person}{Andrew~M Childs}, \bibinfo{person}{Robin Kothari}, {and} \bibinfo{person}{Rolando~D Somma}.} \bibinfo{year}{2017}\natexlab{}.
\newblock \showarticletitle{Quantum algorithm for systems of linear equations with exponentially improved dependence on precision}.
\newblock \bibinfo{journal}{\emph{SIAM J. Comput.}} \bibinfo{volume}{46}, \bibinfo{number}{6} (\bibinfo{year}{2017}), \bibinfo{pages}{1920--1950}.
\newblock


\bibitem[Farhi et~al\mbox{.}(2014)]%
        {farhi2014quantum}
\bibfield{author}{\bibinfo{person}{Edward Farhi}, \bibinfo{person}{Jeffrey Goldstone}, {and} \bibinfo{person}{Sam Gutmann}.} \bibinfo{year}{2014}\natexlab{}.
\newblock \showarticletitle{A quantum approximate optimization algorithm}.
\newblock \bibinfo{journal}{\emph{arXiv preprint}} (\bibinfo{year}{2014}).
\newblock
\urldef\tempurl%
\url{https://arxiv.org/abs/1411.4028}
\showURL{%
\tempurl}


\bibitem[Gily{\'e}n(2019)]%
        {gilyen2019thesis}
\bibfield{author}{\bibinfo{person}{Andr{\'a}s Gily{\'e}n}.} \bibinfo{year}{2019}\natexlab{}.
\newblock \emph{\bibinfo{title}{Quantum singular value transformation \& its algorithmic applications}}.
\newblock \bibinfo{thesistype}{Ph.\,D. Dissertation}. \bibinfo{school}{University of Amsterdam}.
\newblock


\bibitem[Gily{\'e}n et~al\mbox{.}(2019)]%
        {gilyen2019quantum}
\bibfield{author}{\bibinfo{person}{Andr{\'a}s Gily{\'e}n}, \bibinfo{person}{Yuan Su}, \bibinfo{person}{Guang~Hao Low}, {and} \bibinfo{person}{Nathan Wiebe}.} \bibinfo{year}{2019}\natexlab{}.
\newblock \showarticletitle{Quantum singular value transformation and beyond: exponential improvements for quantum matrix arithmetics}. In \bibinfo{booktitle}{\emph{Proceedings of the 51st Annual ACM SIGACT Symposium on Theory of Computing}}, \bibfield{editor}{\bibinfo{person}{Moses Charikar} {and} \bibinfo{person}{Edith Cohen}} (Eds.). \bibinfo{pages}{193--204}.
\newblock


\bibitem[Gleixner et~al\mbox{.}(2016)]%
        {iterative}
\bibfield{author}{\bibinfo{person}{Ambros~M Gleixner}, \bibinfo{person}{Daniel~E Steffy}, {and} \bibinfo{person}{Kati Wolter}.} \bibinfo{year}{2016}\natexlab{}.
\newblock \showarticletitle{Iterative refinement for linear programming}.
\newblock \bibinfo{journal}{\emph{INFORMS Journal on Computing}} \bibinfo{volume}{28}, \bibinfo{number}{3} (\bibinfo{year}{2016}), \bibinfo{pages}{449--464}.
\newblock


\bibitem[Harrow et~al\mbox{.}(2009)]%
        {Harrow2009_quantum}
\bibfield{author}{\bibinfo{person}{Aram~W. Harrow}, \bibinfo{person}{Avinatan Hassidim}, {and} \bibinfo{person}{Seth Lloyd}.} \bibinfo{year}{2009}\natexlab{}.
\newblock \showarticletitle{Quantum algorithm for linear systems of equations}.
\newblock \bibinfo{journal}{\emph{Physical Review Letters}} \bibinfo{volume}{103}, \bibinfo{number}{15} (\bibinfo{year}{2009}).
\newblock
\urldef\tempurl%
\url{https://doi.org/10.1103/PhysRevLett.103.150502}
\showDOI{\tempurl}


\bibitem[Kerenidis and Prakash(2020)]%
        {kerenidisParkas2020_quantum}
\bibfield{author}{\bibinfo{person}{Iordanis Kerenidis} {and} \bibinfo{person}{Anupam Prakash}.} \bibinfo{year}{2020}\natexlab{}.
\newblock \showarticletitle{A quantum interior point method for LPs and SDPs}.
\newblock \bibinfo{journal}{\emph{ACM Transactions on Quantum Computing}} \bibinfo{volume}{1}, \bibinfo{number}{1} (\bibinfo{year}{2020}), \bibinfo{pages}{1--32}.
\newblock
\urldef\tempurl%
\url{https://doi.org/10.1145/3406306}
\showDOI{\tempurl}


\bibitem[Low and Chuang(2019)]%
        {low2019hamiltonian}
\bibfield{author}{\bibinfo{person}{Guang~Hao Low} {and} \bibinfo{person}{Isaac~L Chuang}.} \bibinfo{year}{2019}\natexlab{}.
\newblock \showarticletitle{Hamiltonian simulation by qubitization}.
\newblock \bibinfo{journal}{\emph{Quantum}}  \bibinfo{volume}{3} (\bibinfo{year}{2019}), \bibinfo{pages}{163}.
\newblock


\bibitem[Mohammadisiahroudi et~al\mbox{.}(2023a)]%
        {mohammadisiahroudi2023Accurately}
\bibfield{author}{\bibinfo{person}{Mohammadhossein Mohammadisiahroudi}, \bibinfo{person}{Brandon Augustino}, \bibinfo{person}{Ramin Fakhimi}, \bibinfo{person}{Giacomo Nannicini}, {and} \bibinfo{person}{Tam{\'a}s Terlaky}.} \bibinfo{year}{2023}\natexlab{a}.
\newblock \showarticletitle{Accurately Solving Linear Systems with Quantum Oracles}.
\newblock \bibinfo{journal}{\emph{Tech Report}} (\bibinfo{year}{2023}).
\newblock


\bibitem[Mohammadisiahroudi et~al\mbox{.}(2023b)]%
        {mohammadisiahroudi2023generating}
\bibfield{author}{\bibinfo{person}{Mohammadhossein Mohammadisiahroudi}, \bibinfo{person}{Ramin Fakhimi}, \bibinfo{person}{Brandon Augustino}, {and} \bibinfo{person}{Tamás Terlaky}.} \bibinfo{year}{2023}\natexlab{b}.
\newblock \showarticletitle{Generating linear, semidefinite, and second-order cone optimization problems for numerical experiments}.
\newblock  (\bibinfo{year}{2023}).
\newblock
\showeprint[arxiv]{2302.00711}


\bibitem[Mohammadisiahroudi et~al\mbox{.}(2022)]%
        {mohammadisiahroudi2022efficient}
\bibfield{author}{\bibinfo{person}{Mohammadhossein Mohammadisiahroudi}, \bibinfo{person}{Ramin Fakhimi}, {and} \bibinfo{person}{Tam{\'a}s Terlaky}.} \bibinfo{year}{2022}\natexlab{}.
\newblock \showarticletitle{Efficient use of quantum linear system algorithms in interior point methods for linear optimization}.
\newblock \bibinfo{journal}{\emph{arXiv preprint arXiv:2205.01220}} (\bibinfo{year}{2022}).
\newblock


\bibitem[Mohammadisiahroudi et~al\mbox{.}(2023c)]%
        {IF-QIPMforLO}
\bibfield{author}{\bibinfo{person}{Mohammadhossein Mohammadisiahroudi}, \bibinfo{person}{Ramin Fakhimi}, \bibinfo{person}{Zeguan Wu}, {and} \bibinfo{person}{Tam{\'a}s Terlaky}.} \bibinfo{year}{2023}\natexlab{c}.
\newblock \showarticletitle{An inexact feasible interior point method for linear optimization with high adaptability to quantum computers}.
\newblock \bibinfo{journal}{\emph{arXiv preprint arXiv:2307.14445}} (\bibinfo{year}{2023}).
\newblock


\bibitem[Monteiro et~al\mbox{.}(2004)]%
        {monteiro2004uniform}
\bibfield{author}{\bibinfo{person}{Renato~DC Monteiro}, \bibinfo{person}{Jerome~W O'Neal}, {and} \bibinfo{person}{Takashi Tsuchiya}.} \bibinfo{year}{2004}\natexlab{}.
\newblock \showarticletitle{Uniform boundedness of a preconditioned normal matrix used in interior-point methods}.
\newblock \bibinfo{journal}{\emph{SIAM Journal on Optimization}} \bibinfo{volume}{15}, \bibinfo{number}{1} (\bibinfo{year}{2004}), \bibinfo{pages}{96--100}.
\newblock


\bibitem[Monteiro and O’Neal(2003)]%
        {Monteiro2003_Convergence}
\bibfield{author}{\bibinfo{person}{Renato~DC Monteiro} {and} \bibinfo{person}{Jerome~W O’Neal}.} \bibinfo{year}{2003}\natexlab{}.
\newblock \showarticletitle{Convergence analysis of a long-step primal-dual infeasible interior-point LP algorithm based on iterative linear solvers}.
\newblock \bibinfo{journal}{\emph{Georgia Institute of Technology}} (\bibinfo{year}{2003}).
\newblock
\urldef\tempurl%
\url{http://www.optimization-online.org/DB_FILE/2003/10/768.pdf}
\showURL{%
\tempurl}


\bibitem[Nannicini(2022)]%
        {nannicini2022fast}
\bibfield{author}{\bibinfo{person}{Giacomo Nannicini}.} \bibinfo{year}{2022}\natexlab{}.
\newblock \showarticletitle{Fast quantum subroutines for the simplex method}.
\newblock \bibinfo{journal}{\emph{Operations Research}} (\bibinfo{year}{2022}).
\newblock


\bibitem[O'Neal(2006)]%
        {o2006use}
\bibfield{author}{\bibinfo{person}{Jerome~W O'Neal}.} \bibinfo{year}{2006}\natexlab{}.
\newblock \bibinfo{booktitle}{\emph{The use of preconditioned iterative linear solvers in interior-point methods and related topics}}.
\newblock \bibinfo{publisher}{Georgia Institute of Technology}.
\newblock


\bibitem[Roos et~al\mbox{.}(2005)]%
        {Roos2005_Interior}
\bibfield{author}{\bibinfo{person}{Cornelis Roos}, \bibinfo{person}{Tam{\'a}s Terlaky}, {and} \bibinfo{person}{J-Ph Vial}.} \bibinfo{year}{2005}\natexlab{}.
\newblock \bibinfo{booktitle}{\emph{Interior Point Methods for Linear Optimization}}.
\newblock \bibinfo{publisher}{Springer New York, NY}.
\newblock
\urldef\tempurl%
\url{https://doi.org/10.1007/b100325}
\showDOI{\tempurl}


\bibitem[van Apeldoorn et~al\mbox{.}(2023)]%
        {van2023quantum}
\bibfield{author}{\bibinfo{person}{Joran van Apeldoorn}, \bibinfo{person}{Arjan Cornelissen}, \bibinfo{person}{Andr{\'a}s Gily{\'e}n}, {and} \bibinfo{person}{Giacomo Nannicini}.} \bibinfo{year}{2023}\natexlab{}.
\newblock \showarticletitle{Quantum tomography using state-preparation unitaries}. In \bibinfo{booktitle}{\emph{Proceedings of the 2023 Annual ACM-SIAM Symposium on Discrete Algorithms (SODA)}}. SIAM, \bibinfo{pages}{1265--1318}.
\newblock


\bibitem[Vazquez et~al\mbox{.}(2022)]%
        {vazquez2022enhancing}
\bibfield{author}{\bibinfo{person}{Almudena~Carrera Vazquez}, \bibinfo{person}{Ralf Hiptmair}, {and} \bibinfo{person}{Stefan Woerner}.} \bibinfo{year}{2022}\natexlab{}.
\newblock \showarticletitle{Enhancing the quantum linear systems algorithm using Richardson extrapolation}.
\newblock \bibinfo{journal}{\emph{ACM Transactions on Quantum Computing}} \bibinfo{volume}{3}, \bibinfo{number}{1} (\bibinfo{year}{2022}), \bibinfo{pages}{1--37}.
\newblock


\bibitem[Wossnig et~al\mbox{.}(2018)]%
        {wossnig2018quantum}
\bibfield{author}{\bibinfo{person}{Leonard Wossnig}, \bibinfo{person}{Zhikuan Zhao}, {and} \bibinfo{person}{Anupam Prakash}.} \bibinfo{year}{2018}\natexlab{}.
\newblock \showarticletitle{Quantum linear system algorithm for dense matrices}.
\newblock \bibinfo{journal}{\emph{Physical Review Letters}} \bibinfo{volume}{120}, \bibinfo{number}{5} (\bibinfo{year}{2018}), \bibinfo{pages}{050502}.
\newblock


\bibitem[Wright(1997)]%
        {Wright1997_Primal}
\bibfield{author}{\bibinfo{person}{Stephen~J Wright}.} \bibinfo{year}{1997}\natexlab{}.
\newblock \bibinfo{booktitle}{\emph{Primal-Dual Interior-Point Methods}}.
\newblock \bibinfo{publisher}{SIAM}.
\newblock
\urldef\tempurl%
\url{https://doi.org/10.1137/1.9781611971453}
\showDOI{\tempurl}


\bibitem[Wu et~al\mbox{.}(2023)]%
        {wu2023inexact}
\bibfield{author}{\bibinfo{person}{Zeguan Wu}, \bibinfo{person}{Mohammadhossein Mohammadisiahroudi}, \bibinfo{person}{Brandon Augustino}, \bibinfo{person}{Xiu Yang}, {and} \bibinfo{person}{Tam{\'a}s Terlaky}.} \bibinfo{year}{2023}\natexlab{}.
\newblock \showarticletitle{An inexact feasible quantum interior point method for linearly constrained quadratic optimization}.
\newblock \bibinfo{journal}{\emph{Entropy}} \bibinfo{volume}{25}, \bibinfo{number}{2} (\bibinfo{year}{2023}), \bibinfo{pages}{330}.
\newblock


\bibitem[Ye et~al\mbox{.}(1994)]%
        {Ye1994_iteration}
\bibfield{author}{\bibinfo{person}{Yinyu Ye}, \bibinfo{person}{Michael~J. Todd}, {and} \bibinfo{person}{Shinji Mizuno}.} \bibinfo{year}{1994}\natexlab{}.
\newblock \showarticletitle{An {$O(\sqrt{nL})$}-iteration homogeneous and self-dual linear programming algorithm}.
\newblock \bibinfo{journal}{\emph{Mathematics of Operations Research}} \bibinfo{volume}{19}, \bibinfo{number}{1} (\bibinfo{year}{1994}), \bibinfo{pages}{53--67}.
\newblock
\urldef\tempurl%
\url{https://doi.org/10.1287/moor.19.1.53}
\showDOI{\tempurl}


\bibitem[Zhou and Toh(2004)]%
        {Zhou2004_Polynomiality}
\bibfield{author}{\bibinfo{person}{Guanglu Zhou} {and} \bibinfo{person}{Kim-Chuan Toh}.} \bibinfo{year}{2004}\natexlab{}.
\newblock \showarticletitle{Polynomiality of an inexact infeasible interior point algorithm for semidefinite programming}.
\newblock \bibinfo{journal}{\emph{Mathematical Programming}} \bibinfo{volume}{99}, \bibinfo{number}{2} (\bibinfo{year}{2004}), \bibinfo{pages}{261--282}.
\newblock
\urldef\tempurl%
\url{https://doi.org/10.1007/s10107-003-0431-5}
\showDOI{\tempurl}


\end{thebibliography}

\appendix

\end{document}